\documentclass[reqno,11pt]{amsart}
\usepackage{color,pdflscape}
\usepackage{tikz}
\usepackage[colorlinks]{hyperref}
\usetikzlibrary{shapes}
\usepackage{fullpage}

\usepackage{amssymb,stmaryrd,amsthm,amsmath,bussproofs}
\usepackage[all]{xy}
\usepackage[latin1]{inputenc}
\usepackage{graphics}
\usepackage{float}

\newcommand{\myemph}[1]{\emph{#1}}    

\newcommand{\forces}{\Vdash}
\newcommand{\app}{\texttt{app}} 
\newcommand{\pair}{\texttt{pair}} 
\newcommand{\Jelim}[1]{\textnormal{\texttt{J}}_{#1}}
\newcommand{\id}[1]{\texttt{Id}_{#1}} 
\newcommand{\idn}[2]{\underline{#1}^{#2}}
\newcommand{\judge}[3][]{#2\;\vdash_{#1}\;#3}

\newcommand{\expand}[3][]{#2\;\vartriangleright_{#1}\;#3}

\newcommand{\zero}{\mathtt{0}}
\newcommand{\rr}{\mathtt{r}}
\newcommand{\successor}{\mathtt{S}}
\newcommand{\rec}{\texttt{rec}}
\newcommand{\RSigma}{\texttt{R}}

\newcommand{\rgrph}{\textnormal{\textbf{Graph}}}
\newcommand{\sets}{\textnormal{\textbf{Set}}}
\newcommand{\groupoids}{\textnormal{\textbf{Gpd}}}

\newtheorem{theorem}{Theorem}[section]
\newtheorem{lemma}[theorem]{Lemma}
\newtheorem{proposition}[theorem]{Proposition}

\theoremstyle{definition}

\theoremstyle{remark}
\newtheorem*{remark}{Remark}

\newcommand{\lscott}{[\![}
\newcommand{\rscott}{]\!]}

\UseComputerModernTips

\DeclareMathAlphabet{\mathitbf}{OML}{cmm}{b}{it}
\DeclareMathSymbol{\alpha}{\mathalpha}{letters}{"0B}
\DeclareMathSymbol{\beta}{\mathalpha}{letters}{"0C}
\DeclareMathSymbol{\gamma}{\mathalpha}{letters}{"0D}
\DeclareMathSymbol{\delta}{\mathalpha}{letters}{"0E}
\DeclareMathSymbol{\epsilon}{\mathalpha}{letters}{"0F}
\DeclareMathSymbol{\zeta}{\mathalpha}{letters}{"10}
\DeclareMathSymbol{\eta}{\mathalpha}{letters}{"11}
\DeclareMathSymbol{\theta}{\mathalpha}{letters}{"12}
\DeclareMathSymbol{\iota}{\mathalpha}{letters}{"13}
\DeclareMathSymbol{\kappa}{\mathalpha}{letters}{"14}
\DeclareMathSymbol{\lambda}{\mathalpha}{letters}{"15}
\DeclareMathSymbol{\mu}{\mathalpha}{letters}{"16}
\DeclareMathSymbol{\nu}{\mathalpha}{letters}{"17}
\DeclareMathSymbol{\xi}{\mathalpha}{letters}{"18}
\DeclareMathSymbol{\pi}{\mathalpha}{letters}{"19}
\DeclareMathSymbol{\rho}{\mathalpha}{letters}{"1A}
\DeclareMathSymbol{\sigma}{\mathalpha}{letters}{"1B}
\DeclareMathSymbol{\tau}{\mathalpha}{letters}{"1C}
\DeclareMathSymbol{\upsilon}{\mathalpha}{letters}{"1D}
\DeclareMathSymbol{\phi}{\mathalpha}{letters}{"1E}
\DeclareMathSymbol{\chi}{\mathalpha}{letters}{"1F}
\DeclareMathSymbol{\psi}{\mathalpha}{letters}{"20}
\DeclareMathSymbol{\omega}{\mathalpha}{letters}{"21}
\DeclareMathSymbol{\varepsilon}{\mathalpha}{letters}{"22}
\DeclareMathSymbol{\vartheta}{\mathalpha}{letters}{"23}
\DeclareMathSymbol{\varpi}{\mathalpha}{letters}{"24}
\DeclareMathSymbol{\varrho}{\mathalpha}{letters}{"25}
\DeclareMathSymbol{\varsigma}{\mathalpha}{letters}{"26}
\DeclareMathSymbol{\varphi}{\mathalpha}{letters}{"27}

\newcommand{\mylist}[1]{\mathitbf{#1}}
\newcommand{\nat}{\mathtt{N}}
\newcommand{\cont}{\mathbf{Cont}}
\newcommand{\synt}{\mathbf{SyntGpd}}
\newcommand{\real}[1]{\mathbf{Real}(\mathcal{#1})}
\newcommand{\C}{\mathbb{C}}
\newcommand{\D}{\mathbb{D}}
\begin{document}

\title{Combinatorial realizability models of type theory}

\author{Pieter Hofstra and Michael A. Warren}

\address{Department of Mathematics and Statistics, University of
  Ottawa\\585 King Edward Ave., Ottawa, Ontario K1N 6N5 Canada}
\email{phofstra@uottawa.ca}

\address{School of Mathematics, Institute for Advanced Study\\Einstein
  Dr., Princeton, New Jersey 08540 USA}
\email{mwarren@math.ias.edu}
\thanks{The results of this paper were originally obtained while
  Warren was a postdoctoral fellow in the Department of Mathematics
  and Statistics at the University of Ottawa, where he was supported by
  the Fields Institute.  He is currently supported by the National Science
  Foundation.  In particular, this material is based upon work
  supported by the National Science Foundation under agreement
  No. DMS-0635607. Any opinions, findings and conclusions or
  recommendations expressed in this material are those of the authors
  and do not necessarily reflect the views of the National Science Foundation.}

\date{\today}

\maketitle

\begin{abstract}
We introduce a new model construction for Martin-L\"of intensional type theory, which is sound
and complete for the 1-truncated version of the theory. The model formally combines the
syntactic model with a notion of realizability; it also encompasses the well-known 
Hofmann-Streicher groupoid semantics. As our main application, we use the model to analyse
the syntactic groupoid associated to the type theory generated by a graph $G$, showing that
it has the same homotopy type as the free groupoid generated by $G$. 
\end{abstract}

\tableofcontents

\section{Introduction}

This paper is a contribution to the study of the compelling connections between homotopy theory
and Martin-L\"of's intensional type theory (see~\cite{Awodey:TTH} for a description of this programme). 
We present a new class of models for intensional type theory which allows us to gain insight
into the homotopy-theoretic behaviour of the type theory in a way which is not possible
with other models such as the syntactic model or the Hofmann-Streicher groupoid models.

We call these models \emph{combinatorial realizability models}, because they associate to
the syntactic data of the theory notions of \emph{realizer}, and because realizers of terms of
higher type are defined in terms of realizers of lower type, much in the same way as Kleene 
realizability defines realizers of universally quantified formulae in terms of functions sending
realizers to realizers. (The extension of the notion of realizer to higher types is also closely
resemblant of ---and indeed inspired by--- Tait's technique of logical predicates, see~\cite{Tait:IIFFTI}.)
 Unlike in classical realizability however,
realizers have, a priori, nothing to do with computable functions;
rather, realizers in our models are generally of
a more combinatorial nature. Indeed, in our motivating example, realizers will be edges 
in a suitable graph; they can also be purely syntactical entities. In the limiting case where
realizers are trivial (i.e. where every derivable term is trivially realized) the model reduces
to the syntactic model, or rather a one-dimensional version thereof.

The type theories for which the model construction primarily is designed are
dependent type theories having dependent sums and products with identity types which
are 1-truncated (in the sense that higher identity proofs are forced to be definitional equalities,
see below). The model also works for 0-truncated (i.e., extensional) type theories.
The theories may further be assumed to have a natural number type with the usual 
recursion principle, and may further be extended by axioms postulating new basic types and
terms. 
However, adding axioms involving non-basic types voids the
warranty. We also have not investigated whether the model works in the presence of
W-types and universes.

\subsection{Motivating problem: homotopy types of ML-complexes}

Let us describe in some detail the questions which prompted the investigations reported on here.
Let $G$ be a (directed,
reflexive) graph and consider the theory $\mathbb{T}_{1}[G]$ obtained by augmenting
ordinary intensional Martin-L\"{o}f type theory with the following
data:
\begin{itemize}
\item A new basic type $\ulcorner G\urcorner$;
\item A new basic term $\ulcorner a\urcorner:\ulcorner G\urcorner$ for
  each vertex $a$ of $G$;
\item A new basic term $\ulcorner f\urcorner:\idn{\ulcorner
    G\urcorner}{}(\ulcorner a\urcorner,\ulcorner b\urcorner)$ for each
  edge $f$ with source $a$ and target $b$; 
\item Axioms 
  \begin{prooftree}
    \AxiomC{~}
    \UnaryInfC{$\rr(a)=\ulcorner 1_{a}\urcorner:\idn{\ulcorner
        G\urcorner}{}(\ulcorner a\urcorner,\ulcorner a\urcorner)$}
  \end{prooftree}
  for each vertex $a$ in $G$; and
\item The \emph{1-truncation rule} which states that iterated identity
  types are trivial:
  \begin{prooftree}
    \AxiomC{$p:\idn{\idn{A}{}(a,b)}{}(f,g)$}
    \UnaryInfC{$f=g:\idn{A}{}(a,b)$}
  \end{prooftree}
  for any type $A$ and terms $a,b:A$.
\end{itemize}
(We use the notation $\idn{A}{}(a,b)$\/ rather than $\id{A}(a,b)$\/ to denote the identity type
at $a,b:A$\/ in order to stress the category-theoretic intuition and to make notation for iterated
identity types more palatable.)

It follows from \cite{Hofmann:GITT} that each closed type $A$ in this theory has an
associated groupoid $|A|$ with (definitional equality classes of) closed terms of type
$A$\/ as objects and where the hom-set $|A|(a,b)$\/ consists of closed
terms of type $\idn{A}{}(a,b)$. The composition, identities and inverses of this groupoid
are given by type-theoretic operations; by virtue of the 1-truncation rule the groupoid laws hold
on the nose, and not just up to higher identity terms.

In particular, we may consider the syntactic groupoid 
$|\ulcorner G\urcorner|$\/ (which we simply write
as $|G|$\/ from now on). 
Sending a graph $G$ to the underlying graph of $|G|$
constitutes the object part of a monad $T$ on the category $\rgrph$ of (directed, reflexive) graphs; following
\cite{Awodey:MC}, the algebras are called \mbox{\myemph{1-truncated Martin-L\"{o}f complexes}}.  Understanding the behavior of these
algebras is a first step in the project initiated in \emph{ibid} of
modeling homotopy types using the higher-categorical structures
arising from type theory. 

The first observation is that the theory $\mathbb{T}_{1}[G]$\/ admits an interpretation using the
Hofmann-Streicher groupoid semantics: in order to specify such an interpretation, it suffices to 
interpret the basic data generating the theory. In principle, we can use any groupoid $H$\/ and any
graph morphism $G \to H$\/ to do this, but an obvious choice of $H$\/ is of course the free groupoid
 $\mathcal{F}(G)$\/ on the graph $G$. For the interpretation of the basic terms of the theory we
then may use the inclusion of generators $G \to \mathcal{F}(G)$. This completely determines the model.

By virtue of the interpretation of elimination terms for identity
types, this yields in particular a functor
\[\Psi: |G| \to \mathcal{F}(G). \]
There is also a functor $\Phi:\mathcal{F}(G) \to |G|$\/ 
in the other direction, induced by the universal property of the free groupoid.
 It sends an object of $\mathcal{F}(G)$, i.e.
a vertex $a$\/ of $G$, to the basic term $\ulcorner a \urcorner:\ulcorner G \urcorner$, 
and a formal composite of edges of $G$\/ to the type-theoretic composite. 
By the universal property of $\mathcal{F}(G)$, $\Phi$\/ is actually a section of 
$\Psi$.

We next note that the syntactic groupoid $|G|$\/ is, intuitively speaking, much larger than
$\mathcal{F}(G)$. This is due to the fact that the type theory derives many more terms than
just the generating basic terms coming from the graph $G$. 
For example, if $f$ is
a loop in the graph $G$ on a vertex $a$ and $b$ is any vertex of $G$, the
elimination rule for identity types gives
\begin{prooftree}
  \AxiomC{$\judge{x,y:G,z:\idn{G}{}(x,y)}{G}$}
  \AxiomC{$\judge{x:G}{b:G}$}
  \BinaryInfC{$\Jelim{}([x:G]b,a,a,f):G$}
\end{prooftree}
where the first two hypotheses are simply obtained by weakening.
One would like to know that such ``duplicate'' or ``doppelg\"{a}nger'' terms do nothing
homotopically harmful.  For example, $T(G)$ should have
the same connected components as $G$.  Similarly, from a logical point
of view one would like to know that there are no non-standard terms of
natural number type in $\mathbb{T}_{1}[G]$ in the sense that we would
like to prove that for each term $t$ of natural number type
there exists a numeral $\mathtt{n}$ and a term of type
$\idn{\nat}{}(t,\mathtt{n})$.  (Whether there exist non-standard terms of
natural number type in the presence of Voevodsky's \myemph{univalence
  axiom} is a related question and we expect that the techniques
developed in this paper can be modified to yield a proof of
\emph{Conjecture 1} from \cite{Voevodsky:NSF}.)  

The main application of our realizability model then, is to answer these and related questions.
More concretely, one of the things we shall show is 
that the comparison functor $\Psi:|G| \to \mathcal{F}(G)$\/ 
is in fact an equivalence of groupoids.
This is done by declaring a realizer of a closed term $t:G$\/ to be a morphism $t \to \overline{t}$\/ 
in $|G|$, where $\overline{t}$\/ is a basic term, i.e. a vertex of $G$. These realizers fit together
to form a natural transformation $1_{|G|} \Rightarrow \Phi\Psi$, exhibiting $\mathcal{F}(G)$\/ as
a deformation retract of $|G|$. In particular, this proves that the two have the same homotopy type.
It is shown in~\cite{Awodey:MC} that the techniques developed here can
be used to prove that the category of 1-truncated ML-complexes is
Quillen equivalent to the category of groupoids, and hence that 1-truncated ML-complexes model
homotopy 1-types.

\subsection{Organization of the paper}

Because the actual construction of the model and the proofs are somewhat technical and
lengthy, we begin in Section~\ref{sec:informal} with an informal explanation of
the construction. In particular, we explain what we mean by realizers, and how different
suitable choices of realizers result in models which give us useful information about the 
syntax and the homotopy-theoretic behaviour of the type theory. We also explain in this section
the general setup of the model using categories of families.

In Section~\ref{sec:synt} we develop various syntactical constructions which allow us
to form a syntactic model of the type theory (but in a manner
different from the ordinary term model). In particular, we show how to associate
groupoids and pseudo-functors to types and contexts, and introduce some machinery for
handling weakening and substitution on the level of syntactical groupoids.
The main consequence of the technology introduced is now that the syntactic groupoids and 
pseudo-functors associated to the type theory give rise to a category with families. 

Section~\ref{sec:model} contains the main material. The realizability model itself is an augmentation
of the syntactic model introduced in Section~\ref{sec:synt}, and is 
obtained by gluing in a notion of realizability. We describe the general
structure of this category with families, and then turn to the semantic type formers, dealing with
dependent products, sums, identity types and natural numbers successively. 

Details regarding some of the intuitively clear but technically involved type-theoretic constructions
have been collected in Appendix~\ref{app:synt}; finally, for ease of reference
Appendix~\ref{app:real} summarizes the realizability clauses which can be extracted from the 
model construction in Section~\ref{sec:model}.

\subsection{Notational Conventions}

The formulation of the rules of type theory we use are listed in complete detail in 
the Appendix. When it does not result in confusion,
 we suppress some typing information to reduce clutter. For example,
instead of the cumbersome $\Jelim{[x,y:A,z:\idn{A}{}(x,y)]}([x:A]\phi,a,b,f)$, we simply write
$\Jelim{}(\phi,a,b,f)$\/ (or $\Jelim{}([x:A]\phi,a,b,f)$\/ when we wish to make clear
which variable is bound in this term).

Because we will often be dealing with lists of terms or variables it
will be convenient to introduce a notation for such lists which will
not result in excess clutter.  In particular, we will denote such
lists by using a bold face font.  E.g., the list
$x_{1},\ldots,x_{n}$ of variables will be denoted by $\mylist{x}$.
Similar notation will be employed for lists of similar terms.  For
example, $\rr(\mylist{\alpha})$ denotes the list of terms
$\rr(\alpha_{1}),\ldots,\rr(\alpha_{n})$ where the list
$\mylist{\alpha}$ is understood.  Also, $t[\mylist{a}/\mylist{x}]$
denotes the term $t[a_{1}/x_{1},\ldots,a_{n}/x_{n}]$ and \emph{not} the
simultaneous substitution $t[a_{1},\ldots,a_{n}/x_{1},\ldots,x_{n}]$.
We will often also avoid displaying variables when the context is clear.

\section{Informal description of the model}\label{sec:informal}

This section explains the structure of the model without giving proofs. We first discuss the
parameters in the model construction, namely the notion of realizer of terms of ground type. 
Various different applications of the model construction arise by choosing a suitable notion of realizer.
We then proceed to explain the categorical form which the interpretation takes, and how terms
of higher and dependent types are interpreted. 

\subsection{Realizers and dense terms}

We focus on a theory of the form $\mathbb{T}_1[G]$, where $G$\/ is a graph. 
To specify an interpretation of such a theory, it will suffice to define what is meant
by a realizer of a closed term of type $G$. Moreover, this assignment of realizers to closed
terms has to be functorial, in the sense that when $f:\idn{G}{}(a,b)$\/ is a closed term of 
identity type, we have an operation which sends realizers of $a$\/ to realizers of $b$.
Formally, this data amounts to a presheaf
\[ \forces: |G| \to \sets\]
on the syntactic groupoid $|G|$. We suggestively write $\alpha \forces t:G$\/ instead of 
$\alpha \in \; \forces\!\!\!(t)$. Since the soundness theorem for the interpretation states that
every term has a realizer, we must insist that basic terms have realizers. More precisely, for
a notion of realizer $\forces$\/ to give rise to a sound interpretation, we must ask that there exists,
for each vertex $a$\/ of the graph $G$, a realizer $\alpha_a \forces a:G$, such that
the functorial action of $\forces$\/ respects this, in the sense that for each edge $f:a \to b$\/ in $G$,
$\alpha_a\cdot f = \alpha_b$\/ (where we write $\alpha \cdot f$\/ for the action of $\forces$\/ 
on arrows).
Once this data has been specified, the model associates to each closed
term $t:G$\/ a distinguished realizer.

In our motivating example, we are concerned with
the functors $\Phi:\mathcal{F}(G) \to |G|$\/ and $\Psi:|G| \to \mathcal{F}(G)$\/ 
relating the syntactic groupoid $|G|$\/ on $G$\/ to the free groupoid on $G$. 
The composite $\Phi\Psi:|G| \to |G|$\/ 
will be referred to as the \emph{closure functor}, and will be denoted $t \mapsto \overline{t}$. 
We shall call a term $t$\/ \emph{dense} when its interpretation in the groupoid model based on
the free groupoid $\mathcal{F}(G)$\/ is an identity. In particular, a closed term $f:\idn{G}{}(a,b)$,
regarded as a morphism in the syntactic groupoid $|G|$, is dense whenever its closure is
an identity $\overline{f}=1_{\overline{a}}$. 

With this notation and terminology, we now define a realizer of a closed term $a:G$\/ to be
a dense map $\alpha:a \to \overline{a}$. We must also specify how this assignment is
functorial in $a$: given a morphism $f:\idn{G}{}(a,b)$\/ and a realizer $\alpha:a \to \overline{a}$,
consider the composite
\[ 
\xymatrix{
b \ar[r]^{f^{-1}} & a \ar[r]^{\alpha} & \overline{a} \ar[r]^{\overline{f}} & \overline{b}
} \]
Since the closure of $\alpha$\/ is the identity, it is easily seen that this morphism is dense,
i.e. that it is a realizer of $b$. It is also clear that this is functorial.

What is not a priori obvious however is that all closed terms have realizers. Indeed, this is the
content of the soundness lemma. However, basic terms are trivially realized by the identity.

It is not necessary, for the specification of a realizability model, to state what realizers of terms of
identity types $\idn{G}{}(a,b)$\/ are. Indeed, assuming we have realizers $\alpha:a \to \overline{a}$\/ 
and $\beta:b \to \overline{b}$, we shall simply declare a closed term $f:\idn{G}{}(a,b)$\/ 
to be realized when the square
\[
\xymatrix{
a \ar[r]^f \ar[d]_\alpha & b \ar[d]^\beta\\
\overline{a} \ar[r]_{\overline{f}} & \overline{b}
}
\]
commutes, i.e. when $\beta$\/ is the result of reindexing the realizer $\alpha$\/ along $f$.

\begin{remark}
An obvious generalization suggests itself: instead of taking realizers for identity proofs between terms
to be determined by those of the terms themselves (and the functorial action on realizers) we could
take a realizer to be a higher identity proof witnessing the fact that the above square commutes up 
to equivalence. We conjecture that this is the appropriate generalization of the model to theories
without the truncation rule.
\end{remark}

There are other notions of realizer which may be used, and in order to avoid the impression that
the above example is the only one, we briefly point out a couple of other
possibilities (several other examples can be found in \cite{Awodey:MC}). First, one could consider the
terminal presheaf $|G| \to \sets$\/ which sends each object to a fixed singleton set. In this case,
realizers contain no information, and the resulting model will simply be a 1-truncated syntactic model
(i.e. a model of the theory purely based on syntactic groupoids). Second, more general closure functors
can be used. Any endofunctor $F:|G| \to |G|$\/ gives rise to a model in which a realizer for a term
$a$\/ is a dense map $a \to Fa$, as long as basic terms have realizers. 

A further generalization is possible:

\begin{theorem}
Let $G$\/ be a graph, $H$\/ a groupoid and let $P,Q:|G| \to H$\/ be two functors. 
Suppose furthermore that we are given
a morphism $\alpha_a: P(a) \to Q(a)$\/ for each basic term $a$
such that, for each basic term $f:a \to b$, the following diagram commutes:
\[
\xymatrix{
P(a) \ar[d]_{P(f)} \ar[r]^{\alpha_a} & Q(a) \ar[d]^{Q(f)} \\
P(b) \ar[r]_{\alpha_b} & Q(b).
}
\]
Then there exists a natural transformation $\alpha:P \Rightarrow Q$\/ whose component 
at a basic term $a$\/ is $\alpha_a$.
\end{theorem} 

\begin{proof}
We define a notion of realizability for $G$: let $\forces:|G| \to \sets$\/ be the presheaf which 
sends an object $t:G$\/ to the set $H(P(t),Q(t))$. The functorial action is by conjugation:
given $f:t \to s$\/ and $\tau \forces t:G$, we set $\tau\cdot f =_{\textrm{def}} Q(f)\tau P(f)^{-1}$.
A basic term $a:A$\/ is then realized by the given $\alpha_a$. This gives a realizability model,
the soundness of which gives in particular that each $t:G$\/ has a realizer, and that these
form a natural transformation as desired.
\end{proof}

An application of this more general result appears in~\cite{Awodey:MC}, where it is used to show that
the free 1-truncated ML-complex on a groupoid $G$\/ (i.e. an object in the image of the left adjoint to 
the forgetful functor from ML-complexes to groupoids) is equivalent to $G$.

\subsection{The general setup of the model}

We now turn to the general organization of the model, which augments the syntactic model with
the concept of realizer. We begin by sketching the syntactic structure, and then explain how
the realizers are added into this model.

As above, we shall denote the syntactic groupoid associated to a closed type
$A$\/ by $|A|$. This can be extended to contexts and type judgements: to each well-formed context $\Gamma$\/ we associate a groupoid $|\Gamma|$, and to each type judgement $\Gamma \vdash
T(\mylist{x})$\/ we associate a pseudo-functor 
\[ |\Gamma \vdash T(\mylist{x})|: |\Gamma| \to \groupoids\]
which sends an object $\mylist{c}$\/ of $|\Gamma|$\/ to the syntactic groupoid
$|T(\mylist{c})|$. (The action on morphisms is given by a type-theoretic version of
change-of-base; details will be provided in section~\ref{sec:synt}.)
Then we use the Grothendieck construction to define the syntactic groupoid of the extended context to be
\[ |\Gamma, y:T(\mylist{x})|=\int |\Gamma \vdash T(\mylist{x})|.\]
Thus we obtain a fibration of groupoids
\[ |\Gamma, y:T| \to |\Gamma|. \]

Suppose now that we have a term judgement $\Gamma \vdash t:T$. Then
the interpretation $|\judge{\Gamma}{t:T}|$ of
$t$\/ will be a section of the fibration $|\Gamma, y:T| \to |\Gamma|$, namely the one
which sends an object $\mylist{c}$\/ of $\Gamma$\/ to the object $t(\mylist{c})$. 

The realizability model will interpret the syntax in a similar way. We will denote the
interpretation in the realizability model by $\lscott - \rscott$\/ (where it is assumed we have 
fixed a notion of realizer).

First, consider the basic type $\ulcorner G \urcorner$. By definition, a notion of realizer is
a functor $\mathcal{R}: |G| \to \sets$. Applying the Grothendieck construction to this presheaf, we obtain 
a groupoid
\[ \lscott G \rscott = \int \mathcal{R} \]
which we take to be the interpretation of the type $G$. It naturally comes with a fibration
\[ 
\xymatrix{ \lscott G \rscott \ar[r]^{\pi_G} & |G|.}
\]
This will be a general pattern: the realizability interpretation of any closed type (context) $\Gamma$\/ will 
be a groupoid $\lscott \Gamma \rscott$\/ fibred over $|\Gamma|$:
\[ \xymatrix{ \lscott \Gamma \rscott \ar[r]^{\pi_\Gamma} & |\Gamma|.}\]

It is technically convenient to introduce an intermediate construction: to each context $\Gamma$\/ 
we will also associate a groupoid $\|\Gamma\|$. 
Thus we will have three groupoids $|\Gamma|,\|\Gamma\|$ and
$\lscott\Gamma\rscott$\/ which will fit together in a commutative
diagram
\begin{align}\label{eq:a_nice_triangle}
  \xy
  {\ar(0,15)*+{\lscott\Gamma\rscott};(30,15)*+{\|\Gamma\|}};
  {\ar_{\pi_{\Gamma}}(0,15)*+{\lscott\Gamma\rscott};(15,0)*+{|\Gamma|}};
  {\ar(30,15)*+{\|\Gamma\|};(15,0)*+{|\Gamma|}};
  \endxy
\end{align}
such that each of the components is a fibration.  With this picture in
mind,  the interpretation of a judgement $\judge{\Gamma}{T}$ is a functor
\begin{align*}
  \lscott\judge{\Gamma}{T}\rscott:\|\Gamma,x:T\|\to\sets
\end{align*}
which is to be thought of as sending an object to the set
of realizers of that object; then the extended context is interpreted by Grothendieck construction as
\begin{align*}
  \lscott\Gamma, x:T\rscott & =_{\textrm{def}}
  \int\lscott\judge{\Gamma}{T}\rscott.
\end{align*}
We will also have at each stage of the construction that 
\begin{align*}
  \|\Gamma,x:T\| & =_{\textrm{def}} \int |\judge{\Gamma}{T}|\circ\pi_{\Gamma}
\end{align*}
so that there is always a pullback diagram
\begin{align*}
  \xy
  {\ar(0,15)*+{\|\Gamma,x:T\|};(30,15)*+{|\Gamma,x:T|}};
  {\ar(0,15)*+{\|\Gamma,x:T\|};(0,0)*+{\lscott\Gamma\rscott}};
  {\ar_{\pi_{\Gamma}}(0,0)*+{\lscott\Gamma\rscott};(30,0)*+{|\Gamma|}};
  {\ar(30,15)*+{|\Gamma,x:T|};(30,0)*+{|\Gamma|}};
  \endxy
\end{align*}
in the category of groupoids. 

Putting this together with what was described in
(\ref{eq:a_nice_triangle}) above we have
\begin{align*}
  \xy
  {\ar(0,15)*+{\|\Gamma,x:T\|};(30,15)*+{|\Gamma,x:T|}};
  {\ar(0,15)*+{\|\Gamma,x:T\|};(0,0)*+{\lscott\Gamma\rscott}};
  {\ar_{\pi_{\Gamma}}(0,0)*+{\lscott\Gamma\rscott};(30,0)*+{|\Gamma|}};
  {\ar(30,15)*+{|\Gamma,x:T|};(30,0)*+{|\Gamma|}};
  {\ar(0,30)*+{\lscott\Gamma,x:T\rscott};(0,15)*+{\|\Gamma,x:T\|}};
  {\ar^{\pi_{(\Gamma,x:T)}}@/^1pc/(0,30)*+{\lscott\Gamma,x:T\rscott};(30,15)*+{|\Gamma,x:T|}};
  \endxy
\end{align*}

The intuition to keep in mind here is that the groupoid $\|\Gamma,x:T\|$\/ has objects
$(\mylist{c},\mylist{\gamma},t)$, 
where $\mylist{c}$\/ is an object of the syntactic groupoid $\Gamma$,
where the $\mylist{\gamma}$\/ are realizers for $\mylist{c}$, and 
where $t$\/ is an object of $|\Gamma,x:T|$\/ over $\mylist{c}$. 
Thus in this intermediate groupoid 
there is no information about realizers for $t$\/ included, and this is precisely the difference
between $\|\Gamma,x:T\|$\/ and the full interpretation $\lscott
\Gamma,x:T\rscott$, where an object is of the form
$(\mylist{(c,\gamma)},t,\tau)$, with $\tau$\/ a realizer for $t$.

Indeed, in order to make this
intuition a bit more apparent we will often write
\begin{align*}
  \tau\forces_{\mylist{c},\mylist{\gamma}}t:T(\mylist{c})
\end{align*}
to indicate that $\tau$ is an object of the (discrete) groupoid
$\lscott\judge{\Gamma}{T}\rscott_{(\mylist{c},\mylist{\gamma},t)}$.
When no confusion will result we often drop $\mylist{c}$ from the
subscript and simply write $\tau\forces_{\mylist{\gamma}}t:T(\mylist{c})$.
Functoriality of $\lscott\judge{\Gamma}{T}\rscott$ implies that if
$\tau$ is as above and $(\mylist{h},m):(\mylist{c},\mylist{\gamma},t)\to(\mylist{c}',\mylist{\gamma}',t')$\/ is a morphism in $\|\Gamma,x:T\|$, 
then we have
\begin{align*}
  \tau\cdot(\mylist{h},m)\forces_{\mylist{c}',\mylist{\gamma}'}t':T(\mylist{c}').
\end{align*}

Thus, in order to give the interpretation of a type judgement
$\judge{\Gamma}{T}$ it suffices to provide the following data:
\begin{itemize}
\item We must say what is a realizer
  \begin{align*}
    \tau\forces_{\mylist{\gamma}}t:T(\mylist{c})
  \end{align*}
  for $(\mylist{c},\mylist{\gamma})$ an object of $\lscott\Gamma\rscott$
  and $t:T(\mylist{c})$ a term; and
\item We must give a reindexing action which sends a realizer
  $\tau\forces_{\mylist{\gamma}}t:T(\mylist{c})$ to 
  \begin{align*}
    \tau\cdot(\mylist{h},m)\forces_{\mylist{\gamma}'}t':T(\mylist{c}')
  \end{align*}
  for any arrow
  $(\mylist{h},m):(\mylist{c},\mylist{\gamma},t)\to(\mylist{c}',\mylist{\gamma}',t')$
  in $\|\Gamma,x:T\|$,
  and we must verify functoriality of reindexing.
\end{itemize}

\subsection{Interpretation of terms}

A term $\judge{\Gamma}{t:T}$ will be interpreted as a section
\begin{align*}
  \xy
  {\ar^-{\lscott\judge{\Gamma}{t:T}\rscott}(0,15)*+{\lscott\Gamma\rscott};(30,15)*+{\lscott\Gamma,x:T\rscott}};
  {\ar_{1_{\lscott\Gamma\rscott}}(0,15)*+{\lscott\Gamma\rscott};(15,0)*+{\lscott\Gamma\rscott}};
  {\ar(30,15)*+{\lscott\Gamma,x:T\rscott};(15,0)*+{\lscott\Gamma\rscott}};
  \endxy
\end{align*}
satisfisfying the condition that the following diagram
\begin{align*}
  \xy
  {\ar^-{\lscott\judge{\Gamma}{t:T}\rscott}(0,15)*+{\lscott\Gamma\rscott};(30,15)*+{\lscott\Gamma,x:T\rscott}};
  {\ar_{\pi}(0,15)*+{\lscott\Gamma\rscott};(0,0)*+{|\Gamma|}};
  {\ar^{\pi}(30,15)*+{\lscott\Gamma,x:T\rscott};(30,0)*+{|\Gamma,x:T|}};
  {\ar_-{|\judge{\Gamma}{t:T}|}(0,0)*+{|\Gamma|};(30,0)*+{|\Gamma,x:T|}};
  \endxy
\end{align*}
commutes.  This means that to give the interpretation of a term
$\judge{\Gamma}{t:T}$ it suffices to give, for each object $(\mylist{c},\mylist{\gamma})$ of
$\lscott\Gamma\rscott$, a realizer\footnote{I.e., $t[\mylist{\gamma}]$
  is the realizer part of
  $\lscott\judge{\Gamma}{t:T}\rscott(\mylist{c},\mylist{\gamma})$:
  $\lscott\judge{\Gamma}{t:T}\rscott(\mylist{c},\mylist{\gamma}) = (\mylist{c},\mylist{\gamma},t(\mylist{c}),t[\mylist{\gamma}])$.}
\begin{align*}
  t[\mylist{\gamma}]\forces_{\mylist{\gamma}}t(\mylist{c}):T(\mylist{c}),
\end{align*}
and to prove that 
\begin{align}\label{eq:weak_C2}
  t[\mylist{\gamma}]\cdot\bigl(\mylist{h},t|\mylist{h}\bigr)
  &  = t[\mylist{\gamma}']
\end{align}
for each arrow $\mylist{h}:(\mylist{c},\mylist{\gamma})\to(\mylist{c}',\mylist{\gamma}')$
in $\lscott\Gamma\rscott$. Here, $t|\mylist{h}:t(\mylist{c}) \to t(\mylist{c'})$\/ 
denotes the (second component of) the action of $t$\/ on the morphism $\mylist{h}$, 
see section \ref{sec:synt} for details.

\subsection{Categories with Families}

In order to ensure that the interpretations of contexts, types and terms is independent of derivations,
we organize our semantics as a \emph{Category with Families} (see~\cite{Dybjer:ITT,Hofmann:SSDT} for precise 
definitions, examples, and connections to other notions of semantics for dependent type theories). 
Such a model consists of a category of semantic contexts (which will
serve as interpretations of contexts) and semantic context morphisms (which will serve as interpretations of context morphisms).
In addition, to each semantic context $C$\/ there is associated a set of semantic types $Ty(C)$\/ 
in a contravariant manner; and to each semantic type $T \in Ty(C)$\/ 
there is associated a set of semantic terms $Tm(T)$, again with a contravariant action by context morphisms.
Finally, one requires that this structure admits \emph{comprehension}, in the usual fibrational sense.

A category with families (CwF from now on) interprets the basic calculus of type dependency; 
definitionally equal entities in the type theory are interpreted as set-theoretically equal entities in 
the model. We shall also need to interpret dependent products,
dependent sums, natural numbers and intensional identity types.
A CwF is said to support dependent products (dependent sums, natural
numbers, identity types) if it admits the categorical counterparts
of the type and term constructors for dependent product types
(resp. dependent sums, natural numbers, identity types). 

We will actually construct our model in two stages. First, we show that there is a CwF of syntactic
groupoids: this we obtain by taking the syntactic model (whose underlying category is the category of
contexts) and then transfer this model along the functor which takes a context $\Gamma$\/ and returns its
syntactic groupoid $|\Gamma|$. 

Then we form the gluing of this functor and consider the full subcategory on the cloven fibrations. 
An object will have the form $\lscott \Gamma \rscott \to |\Gamma|$, and this is precisely how
the realizability semantics interprets a context $\Gamma$. Types over $\Gamma$\/ are
then defined to be extended contexts $\Gamma,x:A$, together with a specification of
realizers $\|\Gamma,x:A\| \to \sets$. 

\section{The syntactic structure}\label{sec:synt}

We now describe in detail the syntactic part of the model. First, we introduce sequential $\Jelim{}$-terms,
which are highly useful in describing the type-theoretic constructions corresponding to the iterated
Grothendieck constructions appearing in the interpretation of contexts and types. We then explain
the syntactic interpretation of types and terms and establish soundness.

\subsection{Sequential $\Jelim{}$-terms}

To begin, recall the type-theoretic analogue of (covariantly) reindexing a term $t:B(a)$\/ along
an identity proof $f:\idn{A}{}(a,b)$\/ (where $\judge{x:A}{B(x)}$): we define the term
$t*f:B(b)$\/ as

\begin{prooftree}
  \AxiomC{$\judge{x,y:A,z:\idn{A}{}(x,y)}{B(y)^{B(x)}}$}
  \noLine
  \UnaryInfC{$\judge{x:A}{\lambda v.v:B(x)^{B(x)}}$}
  \noLine  
 \UnaryInfC{$\judge{}{f:\idn{A}{}(a,b)}$}
 \UnaryInfC{$\judge{}{\Jelim{}{(\lambda v.v,a,b,f):B(b)^{B(a)}  } }$}
 \AxiomC{$\judge{}{t:B(a)} $}
\BinaryInfC{$\judge{}{\app(\Jelim{}(\lambda v.v,a,b,f) ,t ):B(b) } $}
\end{prooftree}

Note that in the groupoid interpretation the judgement $\judge{x:A}{B(x)}$\/ takes the form of
a functor from the groupoid interpreting the context $x:A$\/ to the category of groupoids. 
A term $f:\idn{A}{}(a,b)$\/ is interpreted as an arrow $f:a \to b$\/ in the groupoid interpreting $x:A$,
and the term $t*f$\/ is then interpreted as $t \cdot f$, where we have omitted
semantic brackets in order to simplify notation and where $-\cdot f$
denotes the action of the functor interpreting $\judge{x:A}{B(x)}$ on $f$.

Of course, this reindexing action does not just work for closed terms; the above derivation also
works in an ambient context; in particular, both $t$\/ and $a,b$\/ and $f$\/ can 
simply be variables.

Given a context
\begin{align*}
  \Gamma & = \bigl(x_{1}:A_{1},\ldots,x_{n}:A_{n}\bigr)
\end{align*}
there is a new context
$\tilde{\Gamma}$ which consists of variable declarations that we now
describe.  Thinking of the term forming operation $-*-$ described
above as base change the context $\tilde{\Gamma}$ should be thought of
encoding the type theoretic data corresponding to arrows in the
Grothendieck construction of the associated functor.  The first part of
$\tilde{\Gamma}$ has variable declarations 
\begin{align*}
   \bigl(x_{1},y_{1}:A_{1},x_{2}:A_{2}(x_{1}),y_{2}:A_{2}(y_{1}),\ldots,x_{n}:A_{n}(\mylist{x}),y_{n}:A_{n}(\mylist{y})\bigr).
\end{align*}
It then consists of a variable declaration
$z_{1}:\idn{A_{1}}{}(x_{1},y_{1})$.  Given such a variable declaration we
have $x_{2}*z_{1}:A_{2}(y_{1})$.  As such, we may form the next
variable declaration $z_{2}:\idn{A_{2}(y_{1})}{}(x_{2}*z_{1},y_{2})$.
Similarly, given $z_{2}$ we may form
$x_{3}*z_{1}:A_{3}(y_{1},x_{2}*z_{1})$ and then
$(x_{3}*z_{1})*z_{2}:A_{3}(y_{1},y_{2})$.  Accordingly, the next
variable declaration that we add is of the form
\begin{align*}
  z_{3}:\idn{A_{3}(y_{1},y_{2})}{}\bigl((x_{3}*z_{1})*z_{2},y_{3}\bigr).
\end{align*}

So far, the variables $z_1,z_2$\/ and $z_3$\/ may be thought of as representing an
arrow from $x_3$\/ to $y_3$\/ in the groupoid associated to the context
$x_1:A_1,x_2:A_2(x_1),x_3:A_3(x_1,x_2)$. In a diagram:
\[
\xymatrix{
|x_1:A_1,x_2:A_2(x_1),x_3:A_3(x_1,x_2)|\ar[d] 
     && x_3 \ar@{|-->}[d] & x_3*z_1 & (x_3*z_1)*z_2\ar@{|-->}[d] \ar[r]^-{z_3} & y_3 \ar@{|-->}[dl]\\
|x_1:A_1,x_2:A_2(x_1)| \ar[d] 
   && x_2 \ar@{|-->}[d] & x_2*z_1 \ar@{|-->}[d] \ar[r]^-{z_2} & y_2 \ar@{|-->}[dl]\\
|x_1:A_2| && x_1 \ar[r]^{z_1} & y_1 
}
\]
where $z_3$\/ is an arrow in the fibre over $y_2$\/ and $z_2$\/ an arrow in the fibre over $y_1$.

In general, for each $1< m\leq n$,
\begin{align*}
  z_{m}:\idn{A_{m}(y_{1},\ldots,y_{m-1})}{}\bigl((\cdots(x_{m}*z_{1})*\cdots)*z_{m-1},y_{m}\bigr)
\end{align*}
is an arrow in the fibre over $y_{m-1}$.
Then, $\tilde{\Gamma}$ is the collection of all of the variable
declarations
\begin{align*}
  x_{1},\ldots,x_{n},y_{1},\ldots,y_{n},z_{1},\ldots,z_{n}
\end{align*}
of the types described above.  Note that we have 
\begin{align*}
  \tilde{\Gamma}[\mylist{x}/\mylist{y},\rr(\mylist{x})/\mylist{z}]
  & = \Gamma.
\end{align*}
We observe that we have the following derived rule
\begin{prooftree}
  \AxiomC{$\judge{\tilde{\Gamma}}{T}$}
  \noLine
  \UnaryInfC{$\judge{\Gamma}{\varphi:T[\mylist{x}/\mylist{y},\rr(\mylist{x})/\mylist{z}]}$}
  \UnaryInfC{$\judge{\tilde{\Gamma}}{\Jelim{[\judge{\tilde{\Gamma}}{T}]}^{\sigma}([\mylist{x}]\varphi,\mylist{x},\mylist{y},\mylist{z}):T}$}
\end{prooftree}
Here, as in the usual formulation of the elimination rule, the
variables enclosed in square brackets $[\mylist{x}]$ indicate that these variables are bound in
the term $\varphi$.  When no confusion will result these and other
bits of bookkeeping (such as the subscript
$\judge{\tilde{\Gamma}}{T}$) will be omitted.  

Just as an ordinary $\Jelim{}$-term can be regarded as an expansion of a term along an arrow,
a sequential $\Jelim{}$-term is to be thought of as an expansion of a term along an arrow 
in an iterated Grothendieck construction.

These terms satisfy the following conversion rule
\begin{prooftree}
  \AxiomC{$\judge{\tilde{\Gamma}}{T}$}
  \noLine
  \UnaryInfC{$\judge{\Gamma}{\varphi:T[\mylist{x}/\mylist{y},\rr(\mylist{x})/\mylist{z}]}$}
  \UnaryInfC{$\judge{\Gamma}{\Jelim{}^{\sigma}(\varphi,\mylist{x},\mylist{y},\mylist{z})[\mylist{x}/\mylist{y},\rr(\mylist{x})/\mylist{z}]\;=\;\varphi}$}
\end{prooftree}

The explicit definition of these terms and the corresponding parameterized
versions can be found in Section \ref{app:seq}.  We point out that this construction
is not novel: it appeared in~\cite{Gambino:ITWFS}.

\subsection{Groupoids and pseudo-functors associated to types}

We now proceed to associate to each context $\Gamma$\/ a groupoid $|\Gamma|$, and
to each type judgement $\judge{\Gamma}{T}$\/ a pseudo-functor
\[ |\judge{\Gamma}{T}|:|\Gamma| \to \groupoids. \]
This will be done by induction on the length of the context $\Gamma$.

When $\Gamma=()$\/ is the empty context, we set $|()|=1$, the terminal groupoid.
For a closed type judgement $\judge{}{A}$\/ (i.e. when $A$\/ is a closed type) we 
set $|\judge{}{A}|:1 \to \groupoids$\/ to be the pseudo-functor corresponding to the groupoid
$|A|$, the syntactic groupoid of $A$\/ (which, as explained earlier, has definitional equality classes
of closed terms $a:A$\/ as objects and terms $f:\idn{A}{}(a,b)$\/ as its morphisms). 
The context $(x:A)$\/ is then interpreted as the Grothendieck construction of this pseudo-functor,
i.e. we have $|x:A|=_{\textrm{def}}|A|$.

Next, suppose that we have already defined the groupoid $|\Gamma|$\/ associated to 
a valid context $\Gamma$, and that we are given a type
judgement $\judge{\Gamma}{T}$. We wish to define the pseudo-functor 
$|\judge{\Gamma}{T}|:|\Gamma| \to \groupoids$. On an object $\mylist{c}$\/ of $\Gamma$,
we define
\[ |\judge{\Gamma}{T}|_\mylist{c} =_{\textrm{def}} |T(\mylist{c})|. \]
That is, the object $\mylist{c}$\/ is sent to the syntactic groupoid of the closed type $T(\mylist{c})$.
To an arrow $\mylist{h}:\mylist{c} \to \mylist{d}$\/ in $|\Gamma|$\/ we associate the functor
$|T(\mylist{c})|\to|T(\mylist{d})|$ given by
  \begin{align*}
    t:T(\mylist{c}) & \longmapsto
    \bigl(\cdots(t*h_{1})\cdots*h_{n}\bigr):T(\mylist{d})    
  \end{align*}

Note that this definition makes sense because, by hypothesis,
  \begin{align*}
    h_{1} & :\idn{C_{1}}{}(c_{1},d_{1})\\
    h_{2} & :\idn{C_{2}(d_{1})}{}(c_{2}*h_{1},d_{2})\\
    h_{3} &:\idn{C_{3}(d_{1},d_{2})}{}\bigl((c_{3}*h_{1})*h_{2},d_{3}\bigr)\\
    & \vdots\\
    h_{n} &:\idn{C_{n}(d_{1},\ldots,d_{n-1})}{}\bigl((c_{n}*h_{1})\cdots*h_{n-1},d_{n}\bigr).
  \end{align*}

We will usually denote the action of this functor by $-\cdot \mylist{h}$.
Given $\mylist{h}:\mylist{c}\to\mylist{d}$ and $\mylist{k}:\mylist{d}\to\mylist{e}$\/ 
we must construct the coherence
  natural isomorphism $\gamma(\mylist{h},\mylist{k})$ indicated in the following
  diagram:
  \begin{align*}
    \xy
    {\ar^{-\cdot(\mylist{k}\circ\mylist{h})}(0,15)*+{|T(\mylist{c})|};(30,15)*+{|T(\mylist{e})|}};
    {\ar_{-\cdot\mylist{h}}(0,15)*+{|T(\mylist{c})|};(15,0)*+{|T(\mylist{d})|}};
    {\ar_{-\cdot\mylist{k}}(15,0)*+{|T(\mylist{d})|};(30,15)*+{|T(\mylist{e})|}};
    {\ar@{=>}(15,12)*+{};(15,6)*+{}};
    \endxy
  \end{align*}
  For an object $t$ of $|T(\mylist{c})|$ the natural isomorphism
  $\gamma(\mylist{h},\mylist{k})$ has component
  \begin{align*}
    \gamma(\mylist{h},\mylist{k})_{t}:t\cdot(\mylist{k}\circ\mylist{h}) \to (t\cdot\mylist{h})\cdot\mylist{k}
  \end{align*}
  given by the (parameterized) sequential $\Jelim{}$-term
  \begin{align*}
    \Jelim{}^{\sigma}\bigl(\rr(u\cdot\mylist{w}),\mylist{c},\mylist{d},\mylist{h},\mylist{e},\mylist{k},t\bigr):\idn{T(\mylist{e})}{}\bigl(t\cdot(\mylist{k}\circ \mylist{h}),(t\cdot\mylist{h})\cdot\mylist{k}\bigr)
  \end{align*}
 Because $-\cdot 1_{\mylist{c}}$\/ is not just isomorphic but equal to the identity, we do not need
to specify a unit coherence isomorphism. Finally, the coherence laws follow from the 1-truncation
axiom.

As a consequence of the above, we have the following explicit description of composition 
in the groupoid $|\Gamma,x:T|$:
\[ (\mylist{f'},h') \circ \bigl(\mylist{f},h) = (\mylist{f'}\mylist{f}, h'\circ (h\cdot \mylist{f'}) \circ
\gamma(\mylist{f},\mylist{f'})\bigr).\]

We will often denote the second component of this map by $h'\circ_{\mylist{f},\mylist{f'}} h$.

\subsection{Interpretation of terms}\label{subsec:terms}

Now that we have an interpretation of types and contexts, we turn to the interpretation of terms.
For an open term $\judge{\Gamma}{t:T}$, the functor
$|t|:|\Gamma|\to|\Gamma,y:T|$ sends an object $\mylist{c}$ of
$|\Gamma|$ to the object $(\mylist{c},t(\mylist{c}))$ of
$|\Gamma,y:T|$.  An arrow $\mylist{h}:\mylist{c}\to\mylist{d}$
of $|\Gamma|$ is sent by $|t|$ to the arrow
$(\mylist{h},t|\mylist{h}):(\mylist{c},t(\mylist{c}))\to(\mylist{d},t(\mylist{d}))$
where $t|\mylist{h}$ is the term
\begin{prooftree}
  \AxiomC{$\judge{\tilde{\Gamma}}{\idn{T(\mylist{y})}{}(t(\mylist{x})\cdot\mylist{z},t(\mylist{y}))}$}
  \noLine
  \UnaryInfC{$\judge{\Gamma}{\rr(t(\mylist{x})):\idn{T(\mylist{x})}{}(t(\mylist{x}),t(\mylist{x}))}$}
  \UnaryInfC{$\Jelim{}^{\sigma}(\rr(t(\mylist{x})),\mylist{c},\mylist{d},\mylist{h}):\idn{T(\mylist{d})}{}(t(\mylist{c})\cdot\mylist{h},t(\mylist{d}))$}
\end{prooftree}
For a closed term $t:T$ we have the obvious global section $|t|:1\to|x:T|$.

We need to show that $|t|$\/ is actually a functor. Clearly, when $\mylist{h}$\/ is the identity
then by the conversion rule for sequential $\Jelim{}$-terms $t|\mylist{h}=\rr(t(\mylist{c}))$. 
Given composable maps $\xymatrix{\mylist{c} \ar[r]^{\mylist{h}} & \mylist{d}
\ar[r]^{\mylist{k}} & \mylist{e}}$, we need to verify that the following square commutes:
\[ 
\xymatrix{
t(\mylist{c})\cdot \mylist{h}\cdot \mylist{k} \ar[d]_\gamma 
\ar[rr]^{  (t|\mylist{k}) \cdot \mylist{h}} && t(\mylist{d}) \cdot \mylist{k} \ar[d]^{t|\mylist{k}} \\
t(\mylist{c}) \cdot (\mylist{kh}) \ar[rr]_{t|\mylist{hk}} && t(\mylist{e}).
}
\]
However, there is a propositional equality between the two composites (take $\mylist{k}$\/ to 
be the identity so that we get a definitional
 equality $t|\mylist{h}=\rr(t(\mylist{d})) \circ (t |\mylist{h})$,
and then form the appropriate sequential $\Jelim{}$-term). Hence by 1-truncation 
the two are definitionally equal.

\subsection{Weakening}

Consider a weakening inference

\begin{prooftree}
  \AxiomC{$\judge{\Gamma}{A}$}
  \AxiomC{$\judge{\Delta}{}$}
  \BinaryInfC{$\judge{\Gamma,\Delta}{A}$}
\end{prooftree}

Note that there exists a functor
\begin{equation}\label{eq:def_tau}
 \tau_{\Gamma;\Delta;A}:|\Gamma,\Delta,x:A|\to|\Gamma,x:A|
\end{equation}
 given as follows:
\begin{description}
\item[On Objects] For $(\mylist{c},\mylist{d},a)$ an object of
  $|\Gamma,\Delta,x:A|$ we set 
  \begin{align*}
    \tau_{\Gamma;\Delta;A}(\mylist{c},\mylist{d},a) & =_{\textrm{def}} (\mylist{c},a).
  \end{align*}
\item[On Arrows] For an arrow
  $(\mylist{f},\mylist{g},h):(\mylist{c},\mylist{d},a)\to(\mylist{c'},\mylist{d'},a')$
  we first observe that there exists an arrow
  \begin{align*}
    a\dag(\mylist{f},\mylist{g}):a\cdot\mylist{f}=|\judge{\Gamma}{A}|_{\mylist{f}}(a)\to a\cdot(\mylist{f},\mylist{g})=|\judge{\Gamma,\Delta}{A}|_{(\mylist{f},\mylist{g})}(a)
  \end{align*}
  given by the term
  \begin{prooftree}
    \AxiomC{$\judge{\tilde{\Delta}}{\idn{A(\mylist{c}')}{}\bigl(a\cdot\mylist{f},a\cdot(\mylist{f},\mylist{z})\bigr)}$}
    \noLine
    \UnaryInfC{$\judge{\Delta}{\rr(a\cdot\mylist{f}):\idn{A(\mylist{c'})}{}(a\cdot\mylist{f},a\cdot\mylist{f})}$}
    \UnaryInfC{$\Jelim{}^{\sigma}\bigl(\rr(a\cdot\mylist{f}),\mylist{d},\mylist{d'},\mylist{g}\bigr):\idn{A(\mylist{c'})}{}\bigl(a\cdot\mylist{f},a\cdot(\mylist{f},\mylist{g})\bigr)$.}
  \end{prooftree}
  As such, we define
  $\tau_{\Gamma;\Delta;A}(\mylist{f},\mylist{g},h)=\bigl(\mylist{f},h\circ(a\dag(\mylist{f},\mylist{g}))\bigr)$
  where $h\circ a\dag(\mylist{f},\mylist{g})$ is the composite
  indicated in the following diagram:
  \begin{align*}
    \xy
    {\ar^{a\dagger(\mylist{f},\mylist{g})}(0,0)*+{a\cdot\mylist{f}};(30,0)*+{a\cdot(\mylist{f},\mylist{g})}};
    {\ar^{h}(30,0)*+{a\cdot(\mylist{f},\mylist{g})};(60,0)*+{a'}};
    \endxy
  \end{align*}
\end{description}

Clearly the term $a \dagger (\mylist{f},\mylist{g})$\/ is definitionally equal to 
the reflexivity on $a$\/ when $\mylist{f}$\/ and $\mylist{g}$\/ are reflexivities; if $h$\/ is
also a reflexivity then this implies that $\tau_{\Gamma;\Delta;A}(\mylist{f},\mylist{g},h)$\/ 
is one as well. 

To see that
composition is preserved, consider morphisms
$(\mylist{f},\mylist{g},h):(\mylist{c},\mylist{d},a)\to(\mylist{c'},\mylist{d'},a')$\/ and
$(\mylist{f'},\mylist{g'},h'):(\mylist{c'},\mylist{d'},a')\to(\mylist{c''},\mylist{d''},a'')$.
First, we have 
\[ \bigl(\mylist{f'},\mylist{g'},h'\bigr)\bigl(\mylist{f},\mylist{g},h\bigr)=
\bigl(\mylist{f'f},\mylist{g'}\circ_{\mylist{f},\mylist{f'}} \mylist{g}), 
h' \circ_{(\mylist{f},\mylist{g}),(\mylist{f'},\mylist{g'})} h\bigr).\]
Next, observe the coherence isomorphism
\[ \gamma((\mylist{f},\mylist{g}),(\mylist{f'},\mylist{g'})): 
a \cdot (\mylist{f'}\mylist{f},\mylist{g}\circ_{\mylist{f},\mylist{f'}} \mylist{g})) \to 
 a \cdot (\mylist{f},\mylist{g}) \cdot (\mylist{f'},\mylist{g'})\]
makes the following diagram commute
\begin{equation}\label{eq:alpha}
\xymatrix{
a \cdot (\mylist{f'f}) \ar[rr]^{\gamma(\mylist{f},\mylist{f'})} 
\ar[d]_{a \dagger(\mylist{f'f},\mylist{g'} \circ_{\mylist{f},\mylist{f'}} \mylist{g})}
&&
a \cdot \mylist{f} \cdot \mylist{f'} 
\ar[rr]^{(a \cdot \mylist{f}) \dagger(\mylist{f'},\mylist{g'})}
&& 
a \cdot \mylist{f} \cdot (\mylist{f'},\mylist{g'})
  \ar[d]^{(a\dagger(\mylist{f},\mylist{g}))\cdot (\mylist{f'},\mylist{g'})}\\
a \cdot (\mylist{f'f},\mylist{g'} \circ_{\mylist{f},\mylist{f'}} \mylist{g}) 
\ar[rrrr]_{\gamma((\mylist{f},\mylist{g}),(\mylist{f'},\mylist{g'}))}
&&&&
 a \cdot (\mylist{f},\mylist{g})\cdot (\mylist{f'},\mylist{g'})
}
\end{equation}
(To see this, note that the diagram commutes on the nose when $(\mylist{f'},\mylist{g'})$\/ is
the identity; thus there is a sequential $\Jelim{}$-term witnessing the two composites are isomorphic,
and hence equal by 1-truncation.)

Now consider the following diagram:
\[
\xymatrix{
a \cdot (\mylist{f'f}) \ar[r]^{\gamma(\mylist{f},\mylist{f'})} &
a \cdot \mylist{f} \cdot \mylist{f'} \ar[rr]^-{a \dagger(\mylist{f},\mylist{g})\cdot \mylist{f'}}
\ar[d]_{(a \cdot \mylist{f}) \dagger(\mylist{f'},\mylist{g'})}
&& a \cdot (\mylist{f},\mylist{g})\cdot \mylist{f'} \ar[rr]^-{h \cdot \mylist{f'}}
\ar[d]|{(a \cdot (\mylist{f},\mylist{g})) \dagger(\mylist{f'},\mylist{g'}) }
&& a' \cdot \mylist{f'} \ar[d]^{a' \dagger(\mylist{f'},\mylist{g'})} \\
& a \cdot \mylist{f} \cdot (\mylist{f'},\mylist{g'}) 
 \ar[rr]_-{(a\dagger(\mylist{f},\mylist{g}))\cdot (\mylist{f'},\mylist{g'})} 
&& a \cdot (\mylist{f},\mylist{g})\cdot (\mylist{f'},\mylist{g'})
\ar[rr]_-{h \cdot (\mylist{f'},\mylist{g'})} 
&& a' \cdot (\mylist{f'},\mylist{g'}) \ar[r]_-{h'} & a''
}
\]
The two squares commute by naturality of the operation $-\dagger (\mylist{f'},\mylist{g'})$\/ 
(which is established by the same standard argument again). 
Now one way around the diagram corresponds to the composite
\begin{eqnarray*}
 \tau_{\Gamma;\Delta;A}\bigl(\mylist{f'},\mylist{g'},h'\bigr)
\circ \tau_{\Gamma;\Delta;A}\bigl(\mylist{f},\mylist{g},h\bigr) & = &
\bigl(\mylist{f'},h'\circ a'\dagger(\mylist{f'},\mylist{g'})\bigr)
\bigl(\mylist{f},h\circ a'\dagger(\mylist{f},\mylist{g})\bigr)\\
&=&
\bigl(\mylist{f'f},h' \circ a\dagger(\mylist{f'},\mylist{g'}) 
\circ ((h \circ a \dagger(\mylist{f},\mylist{g})) \cdot \mylist{f'}) \circ \gamma(\mylist{f},\mylist{f'}) \bigr)\\
& = & 
\bigl(\mylist{f'f},h' \circ
a\dagger(\mylist{f'},\mylist{g'}) \circ (h \cdot \mylist{f'}) 
\circ (a \dagger(\mylist{f},\mylist{g})\cdot \mylist{f'})\circ \gamma(\mylist{f},\mylist{f'})\bigr)
\end{eqnarray*}

Moreover, the following calculation shows that the other way around the diagram corresponds
to $\tau_{\Gamma;\Delta;A}\bigl((\mylist{f'},\mylist{g'},h') \circ (\mylist{f},\mylist{g},h)  \bigr)$:

\begin{eqnarray*}
\tau_{\Gamma;\Delta;A}\bigl((\mylist{f'},\mylist{g'},h') \circ (\mylist{f},\mylist{g},h)  \bigr)
& = & \tau_{\Gamma;\Delta;A}\bigl(\mylist{f'f},\mylist{g'}\circ_{\mylist{f},\mylist{f'})} \mylist{g}),
h' \circ_{(\mylist{f},\mylist{g}),(\mylist{f'},\mylist{g'})}  h\bigr)\\
& = &  \bigl(\mylist{f'f}, h' \circ_{(\mylist{f},\mylist{g}),(\mylist{f'},\mylist{g'})} h \circ 
a \dagger(\mylist{f'f},\mylist{g'}\circ_{\mylist{f}, \mylist{f'}} \mylist{g}) \bigr)\\
&=& 
\bigl(\mylist{f'f}, h' \circ (h \cdot (\mylist{f'},\mylist{g'})) \circ 
\gamma((\mylist{f},\mylist{g}),(\mylist{f'},\mylist{g'})) \circ
a \dagger(\mylist{f'f},\mylist{g'}\circ_{\mylist{f}, \mylist{f'}} \mylist{g}) \bigr) \\
& = & \bigl(\mylist{f'f}, h' \circ (h \cdot (\mylist{f'},\mylist{g'})) \circ
(a\dagger(\mylist{f},\mylist{g}))\cdot (\mylist{f'},\mylist{g'}) \circ 
(a \cdot \mylist{f}) \dagger(\mylist{f'},\mylist{g'}) \circ \gamma(\mylist{f},\mylist{f'})\bigr)\\
\end{eqnarray*}
where the last step uses the commutativity of~\eqref{eq:alpha}. 
This concludes the proof that $\tau$\/ is functorial.

We also observe:

\begin{lemma}\label{lem:weak_pb}
There is a pullback diagram
\begin{align}\label{eq:a_little_pullback}
  \xy
  {\ar^{\tau_{\Gamma;\Delta;A}}(0,15)*+{|\Gamma,\Delta,x:A|};(30,15)*+{|\Gamma,x:A|}};
  {\ar(0,15)*+{|\Gamma,\Delta,x:A|};(0,0)*+{|\Gamma,\Delta|}};
  {\ar(0,0)*+{|\Gamma,\Delta|};(30,0)*+{|\Gamma|.}};
  {\ar(30,15)*+{|\Gamma,x:A|};(30,0)*+{|\Gamma|.}};
  \endxy
\end{align}
\end{lemma}

\begin{proof}
This is straightforward: given a groupoid $\mathcal{H}$\/ and functors
$L:\mathcal{H} \to |\Gamma,\Delta|$, $M:\mathcal{H} \to |\Gamma,x:A|$\/ making
the outer diagram commute, we define $K:\mathcal{H} \to |\Gamma,\Delta,x:A|$\/ by
$K(u)=(\mylist{c},\mylist{d},a)$, where $L(u)=(\mylist{c},\mylist{d})$\/ and 
$M(u)=(\mylist{c},a)$. For an arrow $l:u \to v$\/ in $\mathcal{H}$, define
\[ K(l)=(\mylist{f},\mylist{g},h \circ (a \dagger(\mylist{f},\mylist{g}))^{-1})\]
where $L(l)=(\mylist{f},\mylist{g})$\/ and $M(l)=(\mylist{f},h)$.
\end{proof}

\subsection{Generalized form}\label{sec:gen_weak}

We shall have the need for a slightly more general form of weakening than the one described above,
namely:

\begin{prooftree}
  \AxiomC{$\judge{\Gamma, \Theta}{}$}
  \AxiomC{$\judge{\Delta}{}$}
  \BinaryInfC{$\judge{\Gamma,\Delta,\Theta}{}$}
\end{prooftree}

In this situation, we wish to define a functor 
\[ \tau_{\Gamma;\Delta;\Theta}:|\Gamma,\Delta,\Theta| \to |\Gamma,\Theta|\]
which fits into a pullback square
\begin{align}\label{eq:a_bigger_pullback}
  \xy
  {\ar^{\tau_{\Gamma;\Delta;\Theta}}(0,15)*+{|\Gamma,\Delta,\Theta|};(30,15)*+{|\Gamma,\Theta|}};
  {\ar(0,15)*+{|\Gamma,\Delta,\Theta|};(0,0)*+{|\Gamma,\Delta|}};
  {\ar(0,0)*+{|\Gamma,\Delta|};(30,0)*+{|\Gamma|.}};
  {\ar(30,15)*+{|\Gamma,\Theta|};(30,0)*+{|\Gamma|.}};
  \endxy
\end{align}
This is done by induction on the length of $\Theta$, where the base case has been addressed
in the previous section. 
For the inductive case, we assume that we have defined 
$\tau_{\Gamma;\Delta;\Theta}:|\Gamma,\Delta,\Theta| \to |\Gamma,\Theta|$\/ 
(fitting into the appropriate pullback square) and now need to define 
$\tau_{\Gamma;\Delta;\Theta,A}:|\Gamma,\Delta,\Theta,x:A| \to |\Gamma,\Theta,x:A|$.
This is done, \emph{mutatis mutandis}, in much the same way as the simpler case in the previous 
section: on objects, the functor sends an object 
$(\mylist{c},\mylist{d},\mylist{e},a)$\/ to $(\mylist{c},\mylist{e},a)$.
For a morphism $(\mylist{f},\mylist{g},\mylist{k},h):(\mylist{c},\mylist{d},\mylist{e},a)
\to (\mylist{c}',\mylist{d}',\mylist{e}',a')$, we note first that by inductive hypothesis we are given
\[ \mylist{e}\dagger (\mylist{f},\mylist{g}): \mylist{e} \cdot \mylist{f} \to \mylist{e}\cdot (\mylist{f},\mylist{g}).\] 
Then there is a canonical term
\[ a \dagger (\mylist{f},\mylist{g},\mylist{k}): a \cdot (\mylist{f},\mylist{k}\cdot \mylist{e}\dagger(\mylist{f},\mylist{g})) \to
a \cdot (\mylist{f},\mylist{g},\mylist{k})\]

(which is defined by an appropriate $\Jelim{}$-term).
Then the action of $\tau_{\Gamma;\Delta;\Theta,A}:|\Gamma,\Delta,\Theta,x:A| \to |\Gamma,\Theta,x:A|$\/ on the morphism $(\mylist{f},\mylist{g},\mylist{k},h)$\/ is defined to be
$(\mylist{f},\mylist{k} \cdot \mylist{e}\dagger(\mylist{f},\mylist{g}),h \cdot a \dagger
(\mylist{f},\mylist{g},\mylist{k}))$, and we obtain a pullback square
\[
\xymatrix{
|\Gamma,\Delta,\Theta,x:A| \ar[r] \ar[d] & |\Gamma,\Theta,x:A| \ar[d] \\
|\Gamma,\Delta,\Theta| \ar[r] & |\Gamma,\Theta|
}
\]
The result then follows from elementary properties of pullbacks.
We note for future reference the following special case:

\begin{lemma}\label{lem:ctxt_weak}
Let $\Gamma$\/ and $\Delta$\/ be valid contexts, and let $\Gamma \vdash A$\/ be
a type judgement. Then the weakening functors $|\Delta,\Gamma| \to |\Gamma|$\/ and
$|\Delta,\Gamma,x:A| \to |\Gamma,x:A|$\/ fit into a pullback diagram
\[
\xymatrix{ 
|\Delta,\Gamma,x:A| \ar[r]\ar[d] & |\Gamma,x:A| \ar[d] \\
|\Delta,\Gamma| \ar[r] & |\Gamma|
}
\]
\end{lemma}

We remark that the substitution functors satisfy a coherence principle: any diagram built up out
of these functors will automatically be commutative. On objects, this is virtually immediate from
the definition of these functors, while on arrows it is a consequence of 1-truncation.

\subsection{Substitution}

We next introduce a family of functors which form the categorical counterpart of substitution.
Consider a substitution instance

\begin{prooftree}
  \AxiomC{$\judge{\Gamma}{a:A}$}
  \AxiomC{$\judge{\Gamma,x:A,\Delta}{}$}
  \BinaryInfC{$\judge{\Gamma,\Delta[a/x]}{}$.}
\end{prooftree}
Given this data, we will define a functor
\begin{equation}\label{eq:sigma}
 \sigma_{\Gamma;a;\Delta}:|\Gamma,\Delta[a/x]|\to|\Gamma,x:A,\Delta|.
\end{equation}

To this end, we first recall from Section~\ref{subsec:terms} that the judgement 
$\judge{\Gamma}{a:A}$\/ gives a functor $|a|:|\Gamma| \to |\Gamma,x:A|$\/ which 
sends an object $\mylist{c}$\/ to $(\mylist{c},a(\mylist{c})$\/ and which acts on morphisms
by sending $\mylist{h}:\mylist{c} \to \mylist{d}$\/ to 
$(\mylist{h},a|\mylist{h})$, where $a|\mylist{h}:a(\mylist{c})\cdot \mylist{h} \to a(\mylist{d})$.

Therefore to define the functor $\sigma_{\Gamma;a;\Delta}$, we set
\[ \sigma_{\Gamma;a;\Delta}(\mylist{c},\mylist{d}) =_{\textrm{def}} (\mylist{c},a(\mylist{c}),\mylist{d}).\]

To define the action on morphisms, we proceed by induction on the length of the context,
first assuming that $\Delta=(v_1:B_1)$. Consider a morphism 
$(\mylist{f},g):(\mylist{c},d_1) \to (\mylist{c'},d_1')$.
Of course the first two components of 
$\sigma_{\Gamma;a;\Delta}(\mylist{f},g)$\/ should be $\mylist{f}$\/ and 
$a|\mylist{f}$. The third component should thus be of the form $d_1 \cdot 
(\mylist{f},a|\mylist{f}) \to d_1'$.

Now $d_1$\/  is an object of $|B_1(\mylist{c},a(\mylist{c}))|$, so that
$d_1 \cdot (\mylist{f},a|\mylist{f})$\/ is an object of $|B_1(\mylist{c'},a(\mylist{c'}))|$.
Note first that there is a comparison isomorphism
\[ d_1 \ddagger \mylist{f}: d_1\cdot (\mylist{f},a|\mylist{f}) \to d_1 \cdot \mylist{f}\]
in $|B_1(\mylist{c'},a(\mylist{c'}))|$, arising by taking the parameterized
sequential $\Jelim{}$-term

\begin{prooftree}
\AxiomC{$ \judge{\tilde{\Gamma},u_1:B_1(\mylist{x},a(\mylist{x}))}
{\idn{B_1(\mylist{x},a(\mylist{x}))}{}(u_1\cdot (\mylist{z},a|\mylist{z}),u_1 \cdot \mylist{z} )  }   $}
\noLine
\UnaryInfC{$\judge{\Gamma,u_1:B_1(\mylist{x},a(\mylist{x}))}
{\rr(u_1):\idn{B_1(\mylist{x},a(\mylist{x}))}{}(u_1,u_1)}  $}
\noLine
\UnaryInfC{$ \mylist{f}:\mylist{c} \to \mylist{c'} $}
\UnaryInfC{$\judge{u_1:  B_1(\mylist{c},a(\mylist{c}))}{\Jelim{}^\sigma(\rr(u_1),\mylist{c},
\mylist{c'},\mylist{f},u_1):
\idn{B_1(\mylist{c},a(\mylist{c}))}{}(u_1\cdot (\mylist{f},a|\mylist{f}),u_1 \cdot \mylist{f} )} $}
\end{prooftree}
(where we schematically denote the variables in $\tilde{\Gamma}$\/ by $\mylist{x},
\mylist{y}$\/ and $\mylist{z}$) and setting $u_1=_{\textrm{def}}d_1$.

\begin{lemma}
The family of morphisms $-\ddagger \mylist{f}$\/ is natural in its first argument.
\end{lemma}

\begin{proof}
Given $g_1:d_1 \to d_1'$, we note that there is a propositional equality in the naturality square
\[
\xymatrix{
d_1 \cdot (\mylist{f},a|\mylist{f}) \ar[rr]^{g_1 \cdot  (\mylist{f},a|\mylist{f}) } 
\ar[d]_{d_1 \ddagger \mylist{f}} && d_1' \cdot  (\mylist{f},a|\mylist{f}) 
\ar[d]^{d_1' \ddagger \mylist{f}} \\
d_1 \cdot \mylist{f} \ar[rr]_{g_1 \cdot \mylist{f}} && d_1' \cdot \mylist{f}
}
\]
Thus by 1-truncation, the square commutes.
\end{proof}

We now set $\sigma_{\Gamma;a;\Delta}(\mylist{f},g_1)=(\mylist{f},a|\mylist{f},
g_1 \circ d_1 \ddagger \mylist{f})$, where the last component is 
the composite
\[ 
\xymatrix{
d_1 \cdot (\mylist{f},a|\mylist{f}) \ar[r]^-{d_1 \ddagger \mylist{f}} 
& d_1 \cdot \mylist{f} \ar[r]^-{g_1} & d_1'
}
\]

\begin{lemma}
As defined, $\sigma_{\Gamma;a;\Delta}:|\Gamma,\Delta[a/x]|\to|\Gamma,x:A,\Delta|$\/
is functorial.
\end{lemma}

\begin{proof}
 Consider, in addition to $(\mylist{f},g_1):
(\mylist{c},d_1) \to (\mylist{c'},d_1')$, another morphism
$(\mylist{f'},g_1'):
(\mylist{c'},d_1') \to (\mylist{c''},d_1'')$, giving rise to a composite
$(\mylist{f'f},g_1'\circ_{\mylist{f},\mylist{f'}} g_1)$.

Consider the diagram
\[
\xymatrix{
d_1 \cdot (\mylist{f},a|\mylist{f}) \cdot (\mylist{f'},a|\mylist{f'})
\ar[rr]^-{d_1\ddagger \mylist{f} \cdot (\mylist{f'},a|\mylist{f'})} 
  &&
d_1 \cdot \mylist{f} \cdot (\mylist{f'},a|\mylist{f'}) 
\ar[rr]^{g_1 \cdot (\mylist{f'},a|\mylist{f'})}
\ar[d]|{(d_1 \cdot \mylist{f}) \ddagger \mylist{f'}}
 &&
d_1' \cdot (\mylist{f'},a|\mylist{f'})
\ar[d]^{d_1' \ddagger \mylist{f'}}
  \\
d_1 \cdot (\mylist{f'f},a| \mylist{f'f}) 
\ar[rr]_{d_1\ddagger \mylist{f'f}} \ar[u]^\gamma
  &&
d_1 \cdot \mylist{f} \cdot \mylist{f'}
\ar[rr]_{g_1 \cdot \mylist{f'}}
 &&
d_1'' \cdot \mylist{f'}
\ar[r]_{g_1'}
 & 
d_1''
}
\]
The vertical left-hand map is the coherence isomorphism for the action of $|\Gamma,x:A \vdash
\Delta|$. The right-hand square commutes by naturality of $\ddagger$. The left-hand square 
can also be seen to commute using the by now familiar argument: 
when $\mylist{f'}$\/ is a reflexivity, the square commutes,
and hence there is a propositional identity, and by 1-truncation, an identity between the two 
composites. 

The composite starting with $\gamma$, going along the top and then down is equal to 
$\sigma_{\Gamma;a;\Delta}(\mylist{f'},g_1') \circ 
\sigma_{\Gamma;a;\Delta}(\mylist{f},g_1)$, while the bottom composite
is $\sigma_{\Gamma;a;\Delta}(\mylist{f'f},g_1' \circ_{\mylist{f'},\mylist{f}} g_1)$.
\end{proof}

Next, we use the same construction (but with extra parameters $d_1,d_1',g_1$) to get 
a comparison morphism
\[ d_2 \ddagger (\mylist{f},g_1): d_2 \cdot (\mylist{f},a|\mylist{f},g_1 \circ (d_1 \ddagger 
\mylist{f})) \to d_2 \cdot (\mylist{f},g_1).\]
Thus we put $\sigma_{\Gamma;a;\Delta}(\mylist{f},g_1,g_2)=
(\mylist{f},a|\mylist{f}, g_1 \circ d_1 \ddagger \mylist{f}, g_2 \circ d_2 \ddagger (\mylist{f},g_1))$.
In general we have a comparison isomorphism
\[ d_{m+1} \ddagger (\mylist{f}, \mylist{g}): 
d_{m+1}\cdot (\mylist{f},a|\mylist{f}, \mylist{g} \circ 
(\mylist{d} \ddagger (\mylist{f},\mylist{g}))) \to d_{m+1} \cdot (\mylist{f},\mylist{g})\]
(where now the list $\mylist{g}=(g_1, \ldots, g_m)$) and therefore we put
\[ \sigma_{\Gamma;a;\Delta}(\mylist{f},\mylist{g}) =_{\textrm{def}}(\mylist{f},a|\mylist{f},
\mylist{g} \circ (\mylist{d} \ddagger (\mylist{f},\mylist{g}))). \]
In much the same manner as in the base case we can now show that $-\ddagger(\mylist{f},\mylist{g})$\/ 
is natural in its first argument, and that $\sigma_{\Gamma;a;\Delta}$, as defined, is functorial.

For future reference, we also collect the following fact about the substitution functors.

\begin{lemma}\label{lem:subst_pb}
The functor $\sigma_{\Gamma;a;\Delta}$\/ fits into a pullback square
\[
\xymatrix{
|\Gamma,\Delta[a/x]| \ar[r]^{\sigma_{\Gamma;a;\Delta}} \ar[d] 
& | \Gamma,x:A,\Delta| \ar[d] \\
|\Gamma| \ar[r]_-{|a|} & | \Gamma,x:A| 
}
\]
Moreover, given an extended context $\Gamma,x:A,\Delta,\Delta'$, we have a pullback square
\[
\xymatrix{
|\Gamma,\Delta[a/x],\Delta'[a/x]| \ar[rr]^-{\sigma_{\Gamma;a;\Delta,\Delta'}} \ar[d] 
&& |\Gamma,x:A,\Delta,\Delta'| \ar[d] \\
|\Gamma,\Delta[a/x]| \ar[rr]^-{\sigma_{\Gamma;a;\Delta}} && | \Gamma,x:A,\Delta|
}
\]
\end{lemma}

\begin{proof}
It is clear that the square commutes. Given a groupoid $\mathcal{H}$\/ and functors
$L: \mathcal{H} \to |\Gamma|$, $M:\mathcal{H} \to |\Gamma,x:A,\Delta|$\/ making
the outer diagram commute, define $K:\mathcal{H} \to |\Gamma,\Delta[a/x]|$\/ by
$K(u)=(\mylist{c},\mylist{d})$, where $L(u)=\mylist{c}$\/ and $M(u)=(\mylist{c},a(\mylist{c}),
\mylist{d})$.  For an arrow $l:u \to v$\/ in $\mathcal{H}$\/ define
\[ K(l) =_{\textrm{def}} 
\bigl(\mylist{f},a|\mylist{f}, \mylist{g} \circ (\mylist{d} \ddagger (\mylist{f},\mylist{g}))\bigr),\]
where
$L(u)=\mylist{f}$, and $K(u)=(\mylist{f},a|\mylist{f},\mylist{g})$.

The second claim is immediate from elementary properties of pullback squares.
\end{proof}

\subsection{Syntactical groupoid CwF}

Our last aim for this section is to organize the groupoids and pseudo-functors associated to contexts
and types into a genuine model of the type theory.

Write $\cont$\/ for the category of contexts and context morphisms. This is the underlying
category of a CwF, the syntactic model of the 1-truncated type theory (see~\cite{Hofmann:SSDT}).
 We shall use the constructions from the previous section to define a functor
$|-|:\cont \to \groupoids$.

To a context $\Gamma=(x_1:T_1,\ldots,x_k:T_k)$\/ 
we associate the syntactic groupoid $|\Gamma|$. Next suppose that we are given another context
 $\Delta$\/ and a context morphism $\mylist{m}:\Delta \to \Gamma$, 
i.e., we have a sequence of derivable term judgements
\begin{eqnarray*}
\Delta &\vdash& m_1:T_1 \\
\Delta &\vdash& m_2:T_2[m_1/x_1] \\
&\vdots&\\
\Delta &\vdash &m_k:T_n[m_1/x_1,\ldots, m_{k-1}/x_{k-1}].\\
\end{eqnarray*}

Then we wish to construct a functor
\[ |\mylist{m}|:|\Delta| \to |\Gamma|. \]
To this end, first note that the terms $m_i$\/ induce sections $|m_i|$\/ of the appropriate
projections. These fit together in the following manner (illustrated for the case $k=3$):
\[
\xymatrix{
|\Delta| \ar[r]^-{|m_3|} \ar[dr]_= & |\Delta,x_3:T_3[m_1,m_2]| \ar[d] \ar[r] &
|\Delta,x_2:T_2[m_1],x_3:T_3[m_1]| \ar[r] \ar[d] & |\Delta,x_1:T_1,x_2:T_2,x_3:T_3| \ar[d] \\
 & |\Delta| \ar[dr]_= \ar[r] ^{|m_2|}& |\Delta,x_2:T_2[m_1]| \ar[r] \ar[d] & |\Delta,x_1:T_1,x_2:T_2| \ar[d]\\
&& |\Delta| \ar[r]^{|m_1|} \ar[dr]_= & |\Delta,x_1:T_1|\ar[d]\\
&&& |\Delta|
}
\]
Here the squares are pullbacks (cf. Lemma~\ref{lem:subst_pb}). Generally, given 
$\mylist{m}:\Delta \to \Gamma$, we denote by $|\overline{\mylist{m}}|$\/ the composite

\[
\xymatrix{
|\Delta| \ar[d]^{|m_k|} \\
 |\Delta,x_k:T_k[m_1,\ldots, m_{k-1}]|
 \ar[d]^{\sigma_{\Delta;m_{k-1};T_k}} \\
 |\Delta,x_{k-1}:T_{k-1}[m_1, \ldots, m_{k-2}],x_k:T_k[m_1, \ldots,m_{k-2}]|
\ar[d] \\
 \vdots \ar[d] \\ 
|\Delta,\Gamma|\\}
\]
Then we define the functor $|\mylist{m}|$\/ associated to the context morphism $\mylist{m}$\/ to
be the composite
\[
\xymatrix{
|\Delta| \ar[r]^-{|\overline{\mylist{m}}|} 
& |\Delta,\Gamma| \ar[r]^{\tau_{\Delta;\Gamma}} & |\Gamma|
}
\]
where $\tau_{\Delta;\Gamma}$\/ is the weakening functor described in section~\ref{sec:gen_weak}.

Just as a single term $\Delta \vdash a:T$\/ induces substitution functors
$\sigma_{\Delta;a;\Theta}:|\Delta,\Theta[a/x]| \to |\Delta,x:A,\Theta|$, a context 
morphism $\mylist{m}:\Delta \to \Gamma$\/ also induces substitution functors
$\sigma_{\Delta;\mylist{m};\Theta}:|\Delta,\Theta[\mylist{m}/\mylist{x}]| \to 
|\Delta,\Gamma,\Theta|$\/ fitting into a pullback square
\[
\xymatrix{
|\Delta,\Theta[\mylist{m}]|\ar[rr]^-{ \sigma_{\Delta;\mylist{m};\Theta}} \ar[d] 
 && |\Delta,\Gamma,\Theta|\ar[d] \\
|\Delta| \ar[rr]_{\overline{\mylist{m}}} && |\Delta,\Gamma|
}
\] 
These functors can also be defined explicitly, in a manner very similar to that of the usual
substitution functors; the details are left to the reader.

\begin{proposition}
The above definitions give a functor $|-|:\cont \to \groupoids$.
\end{proposition}

\begin{proof}
Given two composable context morphisms $\mylist{n}:\Theta \to \Delta$\/ and
$\mylist{m}:\Delta \to \Gamma$, we have, on the one hand, the composite functor
\[
\xymatrix{
|\Theta| \ar[r]^-{|\overline{\mylist{m}}|} & |\Theta,\Delta| \ar[r]^-{\tau_{\Theta;\Delta}}
& |\Delta| \ar[r]^-{|\overline{\mylist{n}}|} & |\Delta,\Gamma| \ar[r]^{\tau_{\Delta;\Gamma}} 
& |\Gamma|
}
\]

On the other hand, we have the functor associated to the composite context morphism
$\mylist{m}[\mylist{n}/\mylist{x}]$, which has the form
\[
\xymatrix{
|\Theta| \ar[rr]^-{|\mylist{m}[\mylist{n}/\mylist{x}]|} && |\Theta,\Gamma| 
\ar[r]^{\tau_{\Theta;\Gamma}} & |\Gamma|.
}
\]

In order to show that these are equal, consider the diagram
\[
\xymatrix{
|\Theta| \ar[r]^-{|\overline{\mylist{n}}|} 
\ar[d]_{|\overline{\mylist{m}[\mylist{n}/\mylist{x}]}|} &
|\Theta,\Delta| \ar[r]^-{\tau_{\Theta;\Delta}} \ar[d]^{|\overline{\mylist{m}}|} 
    & |\Delta| \ar[d]^{|\overline{\mylist{m}}|} \\
|\Theta,\Gamma|  \ar[r]^-{\sigma_{\Theta;\mylist{n};\Gamma}}  \ar[d]
& |\Theta,\Delta,\Gamma| \ar[r]^-{\tau_{\Theta;\Delta,\Gamma}}  \ar[d]
& |\Delta,\Gamma| \ar[r]^-{\tau_{\Delta;\Gamma}} \ar[d] & |\Gamma|\\
|\Theta| \ar[r]_{|\overline{\mylist{n}}|} & |\Theta,\Delta| \ar[r]_{\tau_{\Theta;\Delta}}
& |\Delta|
}
\]
(We have overloaded the notation $|\mylist{m}|$; the middle vertical arrow is the
weakened version.)
All squares are pullbacks by general properties of substitution and weakening functors.
Moreover, the vertical composites are identities. The result now follows from the commutativity
of 
\[
\xymatrix{
|\Theta,\Gamma| \ar[r]^-{\sigma_{\Theta;\mylist{n};\Gamma}}  \ar[d]_{\tau_{\Theta;\Gamma}}
& |\Theta,\Delta,\Gamma| \ar[d]^{\tau_{\Theta;\Delta,\Gamma}} \\
|\Gamma| & |\Delta,\Gamma| \ar[l]^{\tau_{\Delta;\Gamma}}
}
\]
which is a consequence of 1-truncation.
\end{proof}

We next use the functor $|-|:\cont \to \groupoids$\/ to transfer the additional CwF-structure
on the category $\cont$\/ of contexts to the category of groupoids; more precisely, we shall define
a CwF-structure on the image of the functor $|-|$.
We shall write $\synt$\/ for the image of the functor $|-|$. Its objects
and arrows will be referred to as \emph{syntactic groupoids}
 and \emph{syntactic functors}, respectively.

\begin{theorem}
The category $\synt$\/ is the underlying category of a CwF, and the functor
$|-|:\cont \to \synt \hookrightarrow \groupoids$\/ induces an isomorphism of CwFs.
\end{theorem}

Since the functor $|-|$\/ is faithful and injective on objects, it is immediate that its image
can be made into a CwF which is isomorphic to $\cont$. However, as we shall now explain,
more can be said, since the CwF structure on the image can be given directly through
more intuitive constructions.

Given a syntactic context $|\Gamma|$, the collection of \emph{types} over $|\Gamma|$\/ is
that of the syntactic model, i.e. it is the set of types $A$\/ in context $\Gamma$. 
As described in section~\ref{sec:informal}, such a type gives rise to a pseudofunctor $|\Gamma| \to \groupoids$,
whose Grothendieck construction defines the syntactic groupoid $|\Gamma,x:A|$, with
first projection $p:|\Gamma,x:A| \to |\Gamma|$. Then by a \emph{syntactic section} of
$p$\/ we mean a section of $p$\/ which is actually a context morphism. It is easily seen
that such a section corresponds uniquely to the interpretation of a
 definitional equality class of terms $\Gamma \vdash t:A$. 
We thus may define a \emph{term} over $\Gamma \vdash A$\/ to be such a syntactic
section. 

For later use we alse record the following easy lemma, the proof of which is easily extracted
from the calculations done so far:

\begin{lemma}\label{lem:term_subst}
Given a context morphism $\mylist{m}:\Delta \to \Gamma$\/ and a term
$\Gamma \vdash t:A$\/ we have a diagram
\[
\xymatrix{
|\Delta| \ar[d]_{t[\mylist{m}/\mylist{x}]} \ar[r]^{\mylist{m}} & |\Gamma| \ar[d]^{|t|} \\
|\Delta,x:A[\mylist{m}/\mylist{x}]| \ar[d] \ar[r] & |\Gamma,x:A| \ar[d] \\
|\Delta| \ar[r]_{\mylist{m}} & |\Gamma|
}
\]
in which the squares are pullbacks and the vertical composites are the identity. 
\end{lemma}

\section{The model}\label{sec:model}

Throughout this section, we work with a 
fixed theory of the form $\mathbb{T}_1[G]$, where $G$\/ is a graph.
In the previous section, we have organized syntactic groupoids into a model
in the form of a category with families. 
In this section, we extend this model by adding a notion of realizer. We first explain 
the general form the resulting CwF will take, and then investigate the semantic type formers
one by one.

\subsection{The Realizability CwF}

So far we have constructed a category with families $\synt$\/ of syntactic groupoids.
This is a sound and complete model of the 1-truncated type theory, but for us it is only
an intermediate step, as we shall now glue this model with a notion of realizability.
This model shall again take the form of a CwF. For clarity, we stress that the CwF $\synt$,
even though the underlying category is a subcategory of the category of groupoids, has
a CwF structure which does not agree with the Hofmann-Streicher model on the category
of groupoids. 

The underlying category of the realizability CwF will be called $\mathcal{V}$, and it is 
defined as follows:
\begin{description}
\item[Objects] cloven fibrations $\C \to |\Gamma|$, where $\Gamma$\/ is a context and
$\C$\/ is a groupoid.
\item[Morphisms] are pairs $(F,|\mylist{m}|)$\/ forming commutative squares
\[
\xymatrix{
\C \ar[r] \ar[d]_F & |\Gamma| \ar[d]^{|\mylist{m}|}\\
\D \ar[r] & |\Delta|
}
\]
\end{description}

Thus, $\mathcal{V}$\/ is the full subcategory on the gluing of $|-|:\synt \to \groupoids$\/ 
determined by the cloven fibrations. In what follows, we shall often denote a typical object
of $\mathcal{V}$\/ by $\pi_\Gamma:\lscott \Gamma \rscott \to |\Gamma|$, and a typical
morphism from $\pi_\Gamma$\/ to $\pi_\Delta$\/ by
 $(\lscott \mylist{m}\rscott,|\mylist{m}|)$. Of course,
the domain of an object in $\mathcal{V}$\/ is not necessarily determined in any way by its codomain,
but the objects which arise by successive applications of comprehension will be determined completely
by their syntactic part. Anticipating the format of the interpretation even more, we will write
the objects of $\lscott \Gamma \rscott$\/ as $(\mylist{c},\mylist{\gamma})$, where
$\mylist{c}$\/ is an object of the groupoid $|\Gamma|$ and
$\mylist{\gamma}$ is an object in the fiber of $\pi_{\Gamma}$ over $\mylist{c}$.

We note that $\mathcal{V}$\/ has a terminal object, namely the identity morphism on the terminal
groupoid, written $\lscott () \rscott \to |()|$. 

The CwF we now define will be called $\real{R}$, where $\mathcal{R}:|G| \to \sets$\/ is
a basic notion of realizer (see section~\ref{sec:informal} for a discussion).

\begin{description}
\item[Underlying category] the category $\mathcal{V}$;
\item[Semantic Types] Given a semantic context $\pi_\Gamma:\lscott \Gamma \rscott \to |\Gamma|$,
 a \emph{semantic type} over $\pi_\Gamma$\/ 
consists of the following data:
\begin{itemize}
\item A type $\Gamma \vdash A$\/ in context $\Gamma$.
\item A functor $\lscott \Gamma \vdash A\rscott:
\|\Gamma,x:A\| \to \sets$, where $\|\Gamma,x:A\|$\/ 
is the groupoid arising in the following pullback:
\[
\xymatrix{
\|\Gamma,x:A\| \ar[r] \ar[d] & |\Gamma,x:A| \ar[d] \\
\lscott \Gamma\rscott  \ar[r]_{\pi_\Gamma} & |\Gamma|
}
\] 
We denote such a type by $(\Gamma,x:A,\forces)$, where $\forces$\/ is the functor
$\|\Gamma,x:A\| \to \sets$. Instead of 
$\tau \in \/\/  \forces\!\!(\mylist{c},\mylist{\gamma},a)$\/ we
write
\[ \tau \forces_{(\mylist{c},\mylist{\gamma})} a:A(\mylist{c})\]
or, when no confusion will result, simply $\tau\forces_{\mylist{\gamma}}a:A(\mylist{c})$.
\end{itemize}
\item[Comprehension] Given a semantic type $(\Gamma,x:A,\forces)$, its comprehension 
consists of the fibration $\lscott \Gamma,x:A \rscott \to |\Gamma,x:A|$, where the domain 
is the Grothendieck construction of the functor $\forces$. 
Then the fibre of the projection $\lscott \Gamma,x:A\rscott \to \lscott \Gamma \rscott$\/ over an object 
$(\mylist{c},\mylist{\gamma})$\/ consists of objects of the form 
$(\mylist{c},\mylist{\gamma},a,\alpha)$\/  where $(\mylist{c},a)$\/ is an object of $|\Gamma,x:A|$\/ 
and $\alpha \forces_{(\mylist{c},\mylist{\gamma})}a:A$.
\item[Type Substitution] Given a semantic context morphism
 $(\lscott \mylist{m}\rscott,|\mylist{m}|):
(\pi_\Delta:\lscott \Delta \rscott \to |\Delta|) \to
(\pi_\Gamma:\lscott \Gamma \rscott \to |\Gamma|)$\/ and given a semantic type 
$(\Gamma,x:A,\forces)$\/ as above, we consider the diagram
\begin{equation}\label{eq:reind}
\xymatrix{
& \|\Gamma,x:A\| \ar[rr] \ar[dd] && |\Gamma,x:A| \ar[dd] \\
\|\Delta,x:A[\mylist{m}]\| \ar[rr] \ar[dd] \ar@{-->}[ur]^{\hat{\sigma}_{\mylist{m};A}} 
   && |\Delta,x:A[\mylist{m}]| \ar[dd] \ar[ur]_{\sigma_{\mylist{m};A}} \\
& \lscott \Gamma \rscott \ar[rr] && |\Gamma| \\
\lscott \Delta \rscott \ar[rr] \ar[ur]_{\lscott \mylist{m} \rscott} && |\Delta| \ar[ur]_{|\mylist{m}|}
}
\end{equation}
in which the bottom square commutes because 
$(\lscott \mylist{m} \rscott,|\mylist{m}|)$\/ is a morphism,
the right-hand square is a pullback by virtue of the properties of the substitution functors,
and the front and back squares are pullbacks by definition. 
Thus there is a unique mediating
arrow $\hat{\sigma}_{\mylist{m};A}:\|\Delta,x:A[\mylist{m}]\| \to \|\Gamma,x:A\|$\/ 
making the left hand square a pullback.
We then define the composite
\[ 
\xymatrix{
\|\Delta,x:A[\mylist{m}]\| \ar[r]^-{\hat{\tau}_{\mylist{m};A}} 
 & \|\Gamma,x:A\| \ar[r]^-\forces & \sets
}
\]
to be the realizer functor associate to the reindexed type.
Note that as a consequence, we have a pullback square
\begin{equation}\label{eq:reind2}
\xymatrix{
\lscott \Delta,x:A[\mylist{m}]\rscott \ar[r]^-{\overline{\tau}_{\mylist{m};A}}
 \ar[d] & \lscott \Gamma,x:A\rscott \ar[d] \\
\lscott \Delta \rscott \ar[r]_{\lscott \mylist{m}\rscott} & \lscott \Gamma\rscott }
\end{equation}
 
The functor $\overline{\tau}_{\mylist{m};A}$\/ sends an object $(\mylist{d},\mylist{\delta},
a,\alpha)$\/ to $(\mylist{m}[\mylist{d}/\mylist{x}], \mylist{m}[\mylist{\delta}],
a,\alpha)$, where $\lscott m \rscott(\mylist{d},\mylist{\delta})=(\mylist{m}[\mylist{d}/\mylist{x}],
\mylist{m}[\mylist{\delta}])$.

\item[Semantic Terms] Given a type 
$(\Gamma,x:A,\forces)$\/ 
as above, we have the induced projection fibration
 $\lscott \Gamma,x:A\rscott \to \lscott \Gamma\rscott$;
 then a 
 \emph{semantic term} is a section $\lscott s\rscott$\/ of this fibration which
\emph{descends to a syntactic section}, in the sense that there is a syntactic section
$|s|:|\Gamma| \to |\Gamma,x:A|$\/ making the following diagram commute:
\[
\xymatrix{
\lscott \Gamma,x:A\rscott \ar[r] & \|\Gamma,x:A\| \ar[r] & |\Gamma,x:A| \\
& \lscott \Gamma\rscott \ar[ul]^{\lscott s\rscott} \ar[r]_{\pi_\Gamma}& |\Gamma| \ar[u]_{|s|}
}
\]
I.e., a term is simply given by a section of the corresponding display
map in the category $\mathcal{V}$. We will write
\[ \lscott t\rscott(\mylist{c},\mylist{\gamma})=
(\mylist{c},\mylist{\gamma},t(\mylist{c}),t[\mylist{\gamma}])\]
to denote the action of such a section.
\item[Term Substitution] Given a semantic term $(\lscott s \rscott,|s|)$\/ as above, and given 
a context morphism $(\lscott \mylist{m}\rscott,|\mylist{m}|)$\/ from 
$(\pi_\Delta:\lscott \Delta\rscott \to |\Delta|)$\/ to
$(\pi_\Gamma:\lscott \Gamma \rscott\to |\Gamma|)$, 
note first that since the diagram~\eqref{eq:reind2} is a pullback, the section
$\lscott s\rscott$\/ induces a section of 
$\lscott \Delta,x:A[\mylist{m}]\rscott \to \lscott \Delta \rscott$. 
This is easily seen to descend to the induced section of $|\Delta,x:A[\mylist{m}]| \to |\Delta|$.
\end{description}

We first note that while in the above formulation pullbacks are used to define substitution for types
and for terms, these are pullbacks of cloven fibrations, 
so there is no concern about coherence. We have:

\begin{proposition}
Under the above definitions, $\real{R}$\/ is a category with families.
\end{proposition}

The generating type $G$\/ is interpreted in $\real{R}$\/ using the notion of realizer $\mathcal{R}$,
as in
\[
\xymatrix{
\lscott x:G \rscott =\int \mathcal{R} \ar[d] \\
\|x:G\| \ar[r]^= \ar[d] & |x:G|\ar[d] \\
\lscott ()\rscott \ar[r]^= & |()|.
}
\]

For future reference, we also record the following immediate consequence of the definition,
which describes the effect of term substitution in terms of realizers:

\begin{lemma}\label{lem:real_subst}
For a context morphism 
$\lscott \mylist{m}\rscott:\lscott \Delta \rscott \to \lscott \Gamma\rscott$, a type 
$\Gamma \vdash A$\/ and
an object $(\mylist{d},\mylist{\delta})$\/ of $\lscott \Delta\rscott$, we have
\[ \tau \forces_{\mylist{\delta}} a:A[\mylist{m}] \text{ iff } 
\tau \forces_{\mylist{m}[\mylist{\delta}]} a:A[\mylist{m}]\]
\end{lemma}

The remainder of this section is dedicated to the proof of the main result:

\begin{theorem}
Let $G$\/ be a graph, and let $\mathcal{R}:|G| \to \sets$\/ be a notion of realizer
for which basic terms have realizers, stable under reindexing. 
Then the CwF $\real{R}$\/ is a sound and complete model of $\mathbb{T}_1[G]$.
\end{theorem}

Of course, completeness is a trivial consequence of the fact that the model contains a full and faithful
copy of the syntactic model.

\subsection{Semantic product types}

Consider a context $\lscott \Gamma\rscott$, a type $A$\/ over this context with comprehension
$\lscott \Gamma,x:A\rscott$\/ and furthermore a type $B$\/ over this comprehension, giving
a further comprehension $\lscott \Gamma,x:A,y:B(x)\rscott$. Thus we have a diagram
\[
\xymatrix{
\lscott \Gamma,x:A,y:B(x)\rscott \ar[r] \ar[d] & \lscott \Gamma,x:A\rscott \ar[d] \ar[r] 
& \lscott \Gamma\rscott \ar[d] \\
|\Gamma,x:A,y:B(x)| \ar[r] & |\Gamma,x:A| \ar[r] & |\Gamma|.
}
\]
We need to define the semantic dependent product type in context $\lscott \Gamma\rscott$.
This we take to consists first of the syntactic type $\Gamma \vdash \prod_{x:A}B(x)$,
giving rise to the pullback diagram
\[
\xymatrix{
\|\Gamma,v:\prod_{x:A}B(x)\| \ar[r] \ar[d] & |\Gamma,v:\prod_{x:A}B(x)| \ar[d] \\
\lscott \Gamma \rscott \ar[r] & |\Gamma|.
}
\]
Second, we must specify the realizability presheaf 
\[ \|\Gamma,v:\prod_{x:A}B(x)\| \to \sets.\]
To this end, consider an object 
$(\mylist{c},\mylist{\gamma})$ of $\lscott\Gamma\rscott$; we have
associated groupoids
\begin{align*}
\lscott\Gamma,x:A\rscott_{(\mylist{c},\mylist{\gamma})}\text{ and }\lscott\Gamma,x:A,y:B(x)\rscott_{(\mylist{c},\mylist{\gamma})}
\end{align*}
obtained as the fibres of
$\lscott\Gamma,x:A\rscott\to\lscott\Gamma\rscott$ and
$\lscott\Gamma,x:A,y:B(x)\rscott\to\lscott\Gamma\rscott$ over
$(\mylist{c},\mylist{\gamma})$. 
Given an object 
$(\mylist{c},\mylist{\gamma},f)$ of $\|\Gamma,v:\prod_{x:A}B(x)\|$, we define
\begin{align*}
  \phi\forces_{\mylist{\gamma}}f:\prod_{x:A(\mylist{c})}B(\mylist{c},x)
\end{align*}
when $\phi$\/ is a section
\begin{align*}
  \xy
  {\ar^(.4){\phi}(0,15)*+{\lscott
      \Gamma,x:A\rscott_{(\mylist{c},\mylist{\gamma})}};(40,15)*+{\lscott
      \Gamma,x:A,y:B(x)\rscott_{(\mylist{c},\mylist{\gamma})}}};
  {\ar_{=}(0,15)*+{\lscott
      \Gamma,x:A\rscott_{(\mylist{c},\mylist{\gamma})}};(20,0)*+{\lscott
      \Gamma,x:A\rscott_{(\mylist{c},\mylist{\gamma})}}};
  {\ar(40,15)*+{\lscott
      \Gamma,x:A,y:B(x)\rscott_{(\mylist{c},\mylist{\gamma})}};(20,0)*+{\lscott
      \Gamma,x:A\rscott_{(\mylist{c},\mylist{\gamma})}}};
  \endxy
\end{align*}
of the canonical projection satisfying the conditions that
\begin{align*}
  \phi(a,\alpha) & =\bigl(a,\alpha,\app(f,a),\phi\{\alpha\}\bigr),\text{ and}\\
  \phi(m) & = \bigl(m,\app(f,x)|(1_{\mylist{c}},m)\bigr),
\end{align*}
for $\alpha\forces_{\mylist{\gamma}}a:A(\mylist{c})$ and $m:a\to a'$ in
$|A(\mylist{c})|$.
Thus, $\phi$\/ sends realizers $\alpha$\/ of $a$\/ to realizers $\phi\{\alpha\}$\/ of
$\app(f,a)$\/ in a functorial way. In particular, 
$\phi$ must satisfy the
condition that 
\begin{align}\label{eq:dep_prod_coh}
  \phi\{\alpha\}\cdot\bigl(1_{\mylist{c}},m,\app(f,x)|(1_{\mylist{c}},m)\bigr)
  & = \phi\{\alpha\cdot(1_{\mylist{c}},m)\}
\end{align}
for every arrow $m:a\to a'$ in $|A(\mylist{c})|$ and
$\alpha\forces_{\mylist{\gamma}}a:A(\mylist{c})$. 

\subsubsection{Reindexing}

We must also define a functorial action on these realizers. To this end, suppose we are given,
in addition to $\phi$\/ as above, a morphism 
$(\mylist{h},m):(\mylist{c},\mylist{\gamma},f)\to(\mylist{c}',\mylist{\gamma}',f')$
in \mbox{$\|\Gamma,v:\prod_{x:A}B(x)\|$}. 
We must then define a realizer
\begin{align*}
  \phi\cdot(\mylist{h},m)\forces_{\mylist{\gamma}'}f':\prod_{x:A(\mylist{c}')}B(\mylist{c}',x).
\end{align*}

First, if
$\alpha\forces_{\mylist{\gamma}'}a:A(\mylist{c}')$, then
we have, by the functorial action on the realizer $\alpha$,
 that \mbox{$\alpha\cdot(\mylist{h}^{-1},1_{a\cdot
  \mylist{h}^{-1}})\forces_{\mylist{\gamma}}a\cdot
\mylist{h}^{-1}:A(\mylist{c})$}.  But then
\begin{align*}
  \phi\bigl\{\alpha\cdot(\mylist{h}^{-1},1_{a\cdot
    \mylist{h}^{-1}})\bigr\}\forces_{\mylist{\gamma},\alpha\cdot(\mylist{h}^{-1},1_{a\cdot\mylist{h}^{-1}})}\app(f,a\cdot\mylist{h}^{-1}):B(\mylist{c},a\cdot\mylist{h}^{-1}).
\end{align*}
Now, we observe that the diagram
\begin{align*}
  \xy
  {\ar^-{(\mylist{h}^{-1},1_{a\cdot\mylist{h}^{-1}})}(0,15)*+{(\mylist{c}',\mylist{\gamma}',a)};(34,15)*+{(\mylist{c},\mylist{\gamma},a\cdot\mylist{h}^{-1})}};
  {\ar_{1_{(\mylist{c}',a)}}(0,15)*+{(\mylist{c}',\mylist{\gamma}',a)};(17,0)*+{(\mylist{c}',\mylist{\gamma}',a)}};
  {\ar^{(\mylist{h},\iota(\mylist{h})_{a})}(34,15)*+{(\mylist{c},\mylist{\gamma},a\cdot\mylist{h}^{-1})};(17,0)*+{(\mylist{c}',\mylist{\gamma}',a)}};
  \endxy
\end{align*}
commutes, where we denote the coherence isomorphism
\mbox{$(a\cdot\mylist{h}^{-1})\cdot\mylist{h}\to a$} by $\iota(\mylist{h})_{a}$. 
It follows from this that there is a morphism
\begin{align*}
  (\mylist{h},\iota(\mylist{h})_{a}):\bigl(\mylist{c},\mylist{\gamma},a\cdot\mylist{h}^{-1},\alpha\cdot(\mylist{h}^{-1},1_{a\cdot\mylist{h}^{-1}})\bigr)\to\bigl(\mylist{c}',\mylist{\gamma}',a,\alpha\bigr)
\end{align*}
in $\lscott\Gamma,x:A\rscott$.
Moreover, in $\|\Gamma,x:A,y:B(x)\|$ we have an arrow
\begin{equation}\label{eq:zeta}
  \zeta:\biggl(\mylist{c},\mylist{\gamma},a\cdot\mylist{h}^{-1},\alpha\cdot(\mylist{h}^{-1},1_{a\cdot\mylist{h}^{-1}}),\app(f,a\cdot\mylist{h}^{-1})\biggr)\to\bigl(\mylist{c}',\mylist{\gamma}',a,\alpha,\app(f',a)\bigr)
\end{equation}
given by 
\begin{align*}
  \zeta & =_{\textrm{def}} (\mylist{h},\iota(\mylist{h})_{a},\Jelim{}(\rr(\app(v,a)),f\cdot\mylist{h},f',m)\circ\Jelim{}^{\sigma}(\rr(\app(u,x)),\mylist{c},\mylist{c}',\mylist{h},a,f))
\end{align*}
where
$\Jelim{}^{\sigma}(\rr(\app(u,x)),\mylist{c},\mylist{c}',\mylist{h},a,f)$
is the term
\begin{prooftree}
  \AxiomC{$\judge{\tilde{\Gamma},x:A(\mylist{w}),u:\prod_{x:A(\mylist{v})}B(\mylist{v},x)}{\idn{B(\mylist{w},x)}{}\bigl(\app(u,x\cdot\mylist{z}^{-1})\cdot(\mylist{z},\iota(\mylist{z})_{x}),\app(u\cdot\mylist{z},x)\bigr)}$}
  \noLine
  \UnaryInfC{$\judge{\Gamma,x:A(\mylist{v}),u:\prod_{x:A(\mylist{v})}B(\mylist{v},x)}{\rr(\app(u,x)):\idn{B(\mylist{v},x)}{}\bigl(\app(u,x),\app(u,x)\bigr)}$}
  \UnaryInfC{$\Jelim{}^{\sigma}(\rr(\app(u,x)),\mylist{c},\mylist{c}',\mylist{h},a,f):\idn{B(\mylist{c}',a)}{}\bigl(\app(f,a\cdot\mylist{h}^{-1})\cdot(\mylist{h},\iota(\mylist{h})_{a}),\app(f\cdot\mylist{h},a)\bigr)$}
\end{prooftree}
and $\Jelim{}(\rr(\app(v,a)),f\cdot\mylist{h},f',m)$ is the term
\begin{prooftree}
  \AxiomC{$\judge{v,w:\prod_{x:A(\mylist{c}')}B(\mylist{c}',x),z:\idn{\prod_{x:A(\mylist{c}')}B(\mylist{c}',x)}{}(v,w)}{\idn{B(\mylist{c}',a)}{}(\app(v,a),\app(w,a))}$}
  \noLine
  \UnaryInfC{$\judge{v:\prod_{x:A(\mylist{c}')}B(\mylist{c}',x)}{\rr(\app(v,a)):\idn{B(\mylist{c}',a)}{}(\app(v,a),\app(v,a))}$}
  \UnaryInfC{$\Jelim{}(\rr(\app(v,a)),f\cdot\mylist{h},f',m):\idn{B(\mylist{c}',a)}{}(\app(f\cdot\mylist{h},a),\app(f',a))$}
\end{prooftree}
Finally, we then have that
\begin{align*}
  \phi\bigl\{\alpha\cdot(\mylist{h}^{-1},1_{a\cdot
    \mylist{h}^{-1}})\bigr\}\cdot\zeta\forces_{\mylist{c}',\mylist{\gamma}',a,\alpha}\app(f',a):B(\mylist{c}',a).
\end{align*}
As such, we define $\phi\cdot(\mylist{h},m)$ by 
\begin{align*}
  \bigl(\phi\cdot(\mylist{h},m)\bigr)(a,\alpha) & =_{\textrm{def}} \bigl(a,\alpha,\app(f',a),\phi\bigl\{\alpha\cdot(\mylist{h}^{-1},1_{a\cdot
    \mylist{h}^{-1}})\bigr\}\cdot\zeta\bigr)
\end{align*}

We leave to the reader the tedious but straightforward verification that this is indeed
functorial. Moreover, the construction of the dependent product types is stable under substitution:
given a context morphism $(\lscott \mylist{m} \rscott,|\mylist{m}|):
(\lscott \Delta \rscott \to |\Delta) \to
(\lscott \Gamma\rscott \to |\Gamma|)$, we must verify that the reindexing of 
$\lscott \Gamma,v:\prod_{x:A}B(x)\rscott$\/ along $\lscott \mylist{m} \rscott$\/ agrees with the
dependent product of the reindexed 
$\lscott \Delta,x:A[\mylist{m}],y:B[\mylist{m}]\rscott$. It suffices to show that
the diagram
\[
\xymatrix{
\|\Delta,v:\prod_{A[\mylist{m}]}B[\mylist{m}]\| \ar[r] \ar[dr]_{\forces} & 
\|\Gamma,v:\prod_{A}B\| \ar[d]^\forces\\
& \sets
}
\]
commutes. But this is an immediate consequence of the fact that the diagram
\[
\xymatrix{
\lscott \Delta,x:A[\mylist{m}],y:B[\mylist{m}]\rscott \ar[r] \ar[d] 
 & \lscott \Gamma,x:A,y:B \rscott \ar[d] \\
 \lscott \Delta,x:A[\mylist{m}] \rscott \ar[r] & \lscott \Gamma,x:A\rscott
}
\]
is a pullback, as this guarantees, for each object $(\mylist{d},\mylist{\delta})$\/ of 
$\lscott \Delta \rscott$, a bijective correspondence between sections of
$\lscott \Delta,x:A[\mylist{m}],y:B[\mylist{m}]\rscott_{(\mylist{d},\mylist{\delta})} \to 
\lscott \Delta,x:A[\mylist{m}]\rscott_{(\mylist{d},\mylist{\delta})}$\/ and sections of
 $\lscott \Gamma,x:A,y:B\rscott_{(\mylist{m}(\mylist{d}),\mylist{m}[\mylist{\delta}])} \to 
\lscott \Gamma,x:A\rscott_{(\mylist{m}(\mylist{d}),\mylist{m}[\mylist{\delta}])}$.

\begin{proposition}
The CwF $\real{R}$\/ supports dependent products.
\end{proposition}

\begin{proof}
We have detailed the semantic type formation; what remains is to construct the appropriate
abstraction and application terms and show that these have the requisite properties.

First, consider a semantic term $s$\/ of type $(\lscott \Gamma,x:A\rscott, \Gamma,x:A \vdash B(x))$.
Then $s=(\lscott s \rscott, |s|)$, where $\lscott s\rscott:\lscott \Gamma,x:A\rscott \to 
\lscott \Gamma,x:A,y:B(x)\rscott$\/ is the section of the appropriate projection, and 
$|s|$\/ is the corresponding syntactic section.
We must define a term of dependent product type, i.e. a section of the projection
\[ \lscott \Gamma,\prod_{x:A}B(x)\rscott \to \lscott \Gamma \rscott. \]

Given $(\mylist{c},\mylist{\gamma})$
an object of $\lscott \Gamma\rscott$ and $(\mylist{c},\mylist{\gamma},a,\alpha)$ an object of
$\lscott\Gamma,x:A\rscott$ we need to define
$\lambda_{x:A}s[\mylist{\gamma}]$.  In particular, we require a realizer
\begin{align*}
  \bigl(\lambda_{x:A}s[\mylist{\gamma}]\bigr)\{\alpha\}\forces_{\mylist{\gamma},\alpha}\app(\lambda_{x:A}s(\mylist{c}),a):B(\mylist{c},a).
\end{align*}
As such, we define
\begin{align*}
  \bigl(\lambda_{x:A}s[\mylist{\gamma}]\bigr)\{\alpha\} & =_{\textrm{def}} s[\mylist{\gamma},\alpha].
\end{align*}
This satisfies the coherence condition for realizers of terms of
dependent product type since $\lscott f\rscott$ satisfies the
corresponding coherence property.  To see that
$\lscott\lambda_{x:A}s\rscott$ is functorial note
that 
\begin{align*}
  \lambda_{x:A}s[\mylist{\gamma}]\cdot(\mylist{h},\lambda_{x:A}s|\mylist{h})\{\alpha\}
  & =
  \lambda_{x:A}s[\mylist{\gamma}]\bigl\{\alpha\cdot(\mylist{h}^{-1},1_{a\cdot\mylist{h}^{-1}})\bigr\}\cdot\bigl(\mylist{h},\iota(\mylist{h})_{a},\zeta\bigr)\\
  & =
  s[\mylist{\gamma},\alpha\cdot(\mylist{h}^{-1},1_{a\cdot\mylist{h}^{-1}})]\cdot(\mylist{h},\iota(\mylist{h})_{a},s|(\mylist{h},\iota(\mylist{h})_{a})\\
  & = s[\mylist{\gamma}',\alpha]\\
  & = \lambda_{x:A}s[\mylist{\gamma}']\{\alpha\},
\end{align*}
where $\zeta$ is the morphism defined in~\eqref{eq:zeta} and where the third
equation is by functoriality of $\lscott s\rscott$.

Finally, suppose we are given a term $f:\prod_{x:A}B(x)$\/ and a term $a:A$, both in context
$\Gamma$. Then we need to give the application term. For this, suppose given $(\mylist{c},\mylist{\gamma})$ in
$\lscott\Gamma\rscott$, then we have 
\begin{align*}
  f[\mylist{\gamma}]\bigl\{a[\mylist{\gamma}]\bigr\}
\forces_{\mylist{\gamma},a[\mylist{\gamma}]}\app(f(\mylist{c}),a(\mylist{c})):B(\mylist{c},a(\mylist{c}))
\end{align*}
and so we define
\begin{align*}
  \app(f,a)[\mylist{\gamma}] & =_{\textrm{def}} f[\mylist{\gamma}]\bigl\{a[\mylist{\gamma}]\bigr\}.
\end{align*}
That this definition is functorial is then immediate.

Given these definitions, the equations which are required to hold for these application and abstraction
terms are straightforwardly verified.
\end{proof}

\subsection{Semantic identity types}

We now turn to the interpretation of identity types.
Since we work in the 1-truncated
version of the type theory, realizers of terms of identity type will be unique when they exist,
and merely serve as a characteristic function for the graph of the functorial action on realizers,
in the sense that a term $f:\idn{A}{}(a,b)$\/ is realized (relative to realizers $\alpha \forces a:A,
\beta\forces b:A$)\/ if and only if $\alpha \cdot f=\beta$. The only complication is the
 elimination rule, which requires us to construct a realizer for $\Jelim{}$-terms. This is taken
care of by using the fact that there exists a morphism from $\phi(a)$\/ to the
elimination term $\Jelim{}(\phi,a,b,f)$; we may then use the functorial action on realizers to 
turn a realizer of $\phi(a)$\/ into one of $\Jelim{}(\phi,a,b,f)$. 

\subsubsection{Identity type formation} 

Consider a semantic context $\lscott \Gamma \rscott  \to |\Gamma|$, and a type 
$(\Gamma,x:A,\forces)$\/ in context $\Gamma$,
giving a context extension $\lscott \Gamma,x:A \rscott$. In order to describe the identity type 
associated to this data, 
we will first given an explicit description of the groupoid $\|\Gamma,x,y:A,z:\idn{A}{}(x,y)\|$.

\begin{description}
\item[Objects] 
An object is a tuple $(\mylist{c},\mylist{\gamma},a,\alpha,b,\beta,f)$
  such that 
  \begin{itemize}
  \item $(\mylist{c},\mylist{\gamma})$ is an object of $\lscott\Gamma\rscott$;
  \item $a,b$\/ are objects of the groupoid $|A(\mylist{c})|$;
  \item $\alpha\forces_{\mylist{\gamma}}a:A(\mylist{c})$;
  \item
    $\beta\forces_{\mylist{\gamma}}b:A(\mylist{c})$;
    and
  \item $f:\idn{A(\mylist{c})}{}(a,b)$.
  \end{itemize}
Note that it is not required that $\alpha\cdot f = \beta$.
\item[Arrows] An arrow
  $(\mylist{c},\mylist{\gamma},a,\alpha,b,\beta,f)\to(\mylist{c'},\mylist{\gamma'},a',\alpha',b',\beta',f')$
  is a tuple $(\mylist{h},m,n,l)$ such that
  \begin{itemize}
  \item
    $\mylist{h}:(\mylist{c},\mylist{\gamma})\to(\mylist{c'},\mylist{\gamma'})$\/ is a morphism 
    in $\lscott\Gamma\rscott$;
  \item $m:a\cdot\mylist{h}\to a'$ in is a morphism in 
$|A(\mylist{c})|$ such that
    \begin{align*}
      \alpha\cdot(\mylist{h},m)=\alpha'\forces_{\mylist{\gamma}'}a':A(\mylist{c}');
    \end{align*}
  \item $n$ is an arrow
    \begin{align*}
      b\cdot(\mylist{h},m)\;=\;|\judge{\Gamma,x:A}{A}|_{(\mylist{h},m)}(b)\to b'
    \end{align*}
    in $|A(\mylist{c})|$ such that
    \begin{align*}
      \beta\cdot(\mylist{h},m,n)=\beta'\forces_{\mylist{\gamma'}} b':A(\mylist{c'}),
    \end{align*}
    where we note that 
    \begin{align*}
      \beta\cdot(\mylist{h},m,n)\;=\;\lscott\judge{\Gamma,x:A}{A}\rscott_{(\mylist{h},m,n)}(\beta)\;=\;\lscott\judge{\Gamma}{A}\rscott_{\hat{\tau}_{\Gamma;x:A;A}(\mylist{h},m,n)}(\beta);\text{ and}
    \end{align*}
  \item $l$ is a propositional equality witnessing the equation (via 1-truncation these
    are the same)
    \begin{align*}
      f\cdot(\mylist{h},m,n)\;=\;|\judge{\Gamma,x,y:A}{\idn{A}{}(x,y)}|_{(\mylist{h},m,n)}(f)
      & = f'.
    \end{align*}
  \end{itemize}
\end{description}

Next, we define the functor $\lscott\Gamma,x,y:A \vdash \idn{A}{}(x,y)\rscott$\/ 
assigning realizers to identity terms. Given an object $(\mylist{c},\mylist{\gamma},
a,\alpha,b,\beta,f)$, we set
\[ \forces_{\mylist{\gamma},\alpha,\beta}f:\idn{A(\mylist{c})}{}(a,b)
\quad\text{ iff }\quad\alpha\cdot(1_{\mylist{c}},f)=\beta.\]
Thus, regarded as a presheaf, this functor is a \emph{subterminal} presheaf, i.e. a subpresheaf
of the presheaf with constant value 1. 
When we need a name for the (by definition unique) realizer of $f$\/ we shall use the 
symbol $\star$.
Note that when $f$\/ is realized, then in particular 
\begin{align*}
  (1_{\mylist{c}},f) : (\mylist{c},\mylist{\gamma},a,\alpha)\to(\mylist{c},\mylist{\gamma},b,\beta)
\end{align*}
is an arrow in $\lscott\Gamma, x:A\rscott$.

We must check functoriality, which in this case amounts to showing that for an arrow

\begin{align*}
  (\mylist{h},m,n,l):(\mylist{c},\mylist{\gamma},a,\alpha,b,\beta,f)\to(\mylist{c'},\mylist{\gamma'},a',\alpha',b',\beta',f')
\end{align*}
in $\|\judge{\Gamma,x,y:A}{\idn{A}{}(x,y)}\|$,

\begin{align*}
  \forces_{\mylist{\gamma},\alpha,\beta}f:\idn{A(\mylist{c})}{}(a,b) &
  \text{ implies }\forces_{\mylist{\gamma}',\alpha',\beta'}f':\idn{A(\mylist{c}')}{}(a',b').
\end{align*} 

So assume that $\forces_{\mylist{\gamma},\alpha,\beta}f:\idn{A(\mylist{c})}{}(a,b)$\/ holds.
First, we observe that the diagram

\begin{align*}
  \xy
  {\ar^{(\mylist{h},m)}(0,20)*+{(\mylist{c},\mylist{\gamma},a)};(50,20)*+{(\mylist{c'},\mylist{\gamma'},a')}};
  {\ar_{(1_{\mylist{c}},f)}(0,20)*+{(\mylist{c},\mylist{\gamma},a)};(0,0)*+{(\mylist{c},\mylist{\gamma},b)}};
  {\ar_{\hat{\tau}_{\Gamma;x:A;A}(\mylist{h},m,n)}(0,0)*+{(\mylist{c},\mylist{\gamma},b)};(50,0)*+{(\mylist{c'},\mylist{\gamma'},b')}};
  {\ar_{(1_{\mylist{c'}},f\cdot(\mylist{h},m,n))}@/_1pc/(50,20)*+{(\mylist{c'},\mylist{\gamma'},a')};(50,0)*+{(\mylist{c'},\mylist{\gamma'},b')}};
  {\ar^{(1_{\mylist{c'}},f')}@/^1pc/(50,20)*+{(\mylist{c'},\mylist{\gamma'},a')};(50,0)*+{(\mylist{c'},\mylist{\gamma'},b')}};
  {(50,10)*+{=}};
  \endxy
\end{align*}
in $\|\judge{\Gamma}{A}\|$ commutes.  Thus, to see that 
$\forces_{\mylist{\gamma}',\alpha',\beta'}f':\idn{A(\mylist{c})}{}(a',b')$
we reason as follows:
\begin{align*}
  \alpha'\cdot(1_{\mylist{c'}},f') & =
  \bigl(\alpha\cdot(\mylist{h},m)\bigr)\cdot(1_{\mylist{c'}},f')\\
  & =
  \bigl(\alpha\cdot(1_{\mylist{c}},f)\bigr)\cdot\hat{\tau}_{\Gamma;x:A;A}(\mylist{h},m,n)\\
  & = \beta\cdot\hat{\tau}_{\Gamma;x:A;A}(\mylist{h},m,n)\\
  & = \beta'.
\end{align*}

Finally, we must verify that this construction is stable under reindexing along context morphisms.
This amounts to verifying that, for a context morphism $\lscott \mylist{m} \rscott:\lscott \Delta
\rscott \to \lscott \Gamma\rscott$, the diagram
\[
\xymatrix{
\|\Delta,x,y:A[\mylist{m}],z:\idn{A[\mylist{m}]}{}(x,y)\| \ar[rr] \ar[drr]_\forces
&& \|\Gamma,x,y:A,z:\idn{A}{}(x,y)\| \ar[d]^\forces\\
&& \sets
}
\]
commutes. This, however, is an immediate consequence of the definition of the realizability 
relation and the fact that the horizontal map induced by the context morphism $\lscott \mylist{m}
\rscott$\/ preserves the functorial action on the realizers.

\subsubsection{Introduction rule}

Consider next the term formation rule

\begin{prooftree}
  \AxiomC{$\judge{\Gamma}{A}$}
  \UnaryInfC{$\judge{\Gamma,x:A}{\rr(x):\idn{A}{}(x,x)}$}
\end{prooftree}
Given an object $(\mylist{c},\mylist{\gamma},a,\alpha)$ of
$\lscott\Gamma, x:A\rscott$ we must show that 
\begin{align*}
  \forces_{\mylist{\gamma},\alpha}\rr(a):\idn{A(\mylist{c})}{}(a,a).
\end{align*}
This holds if and only if 
\begin{align*}
  \forces_{\mylist{\gamma},\alpha,\alpha}\rr(a):\idn{A(\mylist{c})}{}(a,a)
\end{align*}
which holds if and only if 
$\alpha\cdot(1_{\mylist{c}},\rr(a))=\alpha$.  But the latter is
trivially true.

\subsubsection{Elimination rule}

Next, we treat the elimination rule
\begin{prooftree}
  \AxiomC{$\judge{\Gamma,x,y:A,z:\idn{A}{}(x,y)}{B(x,y,z)}$}
  \noLine
  \UnaryInfC{$\judge{\Gamma,x:A}{\varphi:B(x,x,\rr(x))}$}
  \UnaryInfC{$\judge{\Gamma,x,y:A,z:\idn{A}{}(x,y)}{\Jelim{}(\varphi,x,y,z):B(x,y,z)}$}
\end{prooftree}
We assume given the interpretation of the
judgement $\judge{\Gamma,x,y:A,z:\idn{A}{}(x,y)}{B(x,y,z)}$, and, given an object
$(\mylist{c},\mylist{\gamma},a,\alpha,b,\beta,f,\star)$\/  of
$\lscott \Gamma,x,y:A,z:\idn{A}{}(x,y)\rscott$,  a realizer
$\varphi[\mylist{\gamma},\alpha]
\forces_{\mylist{\gamma},\alpha}\varphi(a):B(a,a,\rr(a))$.

We need to define a realizer 
$\xi\forces_{\mylist{\gamma},\alpha,\beta,\star}\Jelim{}(\varphi,a,b,f):B(a,b,f)$.
Actually, we shall define a functor which sends realizers $\varphi[\mylist{\gamma},\alpha]$\/ 
to realizers $\xi$\/ as above.

So consider an object
$(\mylist{c},\mylist{\gamma},a,\alpha,b,\beta,f,\star)$\/  of
$\lscott \Gamma,x,y:A,z:\idn{A}{}(x,y)\rscott$. We begin by constructing a map
\begin{align*}
  (\mylist{c},\mylist{\gamma},a,\alpha,a,\alpha,1_{a},\star) \to (\mylist{c},\mylist{\gamma},a,\alpha,b,\beta,f,\star)
\end{align*}
in $\lscott \Gamma,x,y:A,z:\idn{A}{}(x,y)\rscott$.  For this, observe that there is a term
\begin{align*}
  \Jelim{}(\rr(\rr(x)),a,b,f) :\idn{\idn{A(\mylist{c})}{}(a,b)}{}\bigl(1_{a}\cdot(1_{a},f),f\bigr)
\end{align*}
and therefore a morphism
\begin{align*}
(1_{\mylist{c}},1_{a},f,\Jelim{}(\rr(\rr(x)),a,b,f)):(\mylist{c},\mylist{\gamma},a,\alpha,a,\alpha,1_{a})\to(\mylist{c},\mylist{\gamma},a,\alpha,b,\beta,f)
\end{align*}
in $\|x,y:A,z:\idn{A}{}(x,y)\|$, where we have used the fact
that
\begin{align*}
  \alpha\cdot(1_{\mylist{c}},1_{a},f) & =
  \alpha\cdot\hat{\tau}_{\Gamma;x:A;A}(1_{\mylist{c}},1_{a},f)\\
  & = \alpha\cdot(1_{\mylist{c}},f)\\
  & = \beta.
\end{align*}

So we have constructed a morphism
\begin{align*}
  (1_{\mylist{c}},1_{a},f,\Jelim{}(\rr(\rr(x)),a,b,f)):(\mylist{c},\mylist{\gamma},a,\alpha,a,\alpha,1_{a},\star)\to(\mylist{c},\mylist{\gamma},a,\alpha,b,\beta,f,\star).
\end{align*}
Then we observe that
\begin{align*}
  \Jelim{}\bigl( \rr(\varphi(x)),a,b,f\bigr):\idn{B(a,b,f)}{}\bigl(\varphi(a)\cdot\bigl(1_{\mylist{c}},1_{a},f,\Jelim{}(\rr(\rr(x)),a,b,f)\bigr),\Jelim{}(\varphi,a,b,f)\bigr),
\end{align*}
and therefore we have an arrow
\begin{align*}
  \expand{f}{\varphi}:(\mylist{c},\mylist{\gamma},a,\alpha,a,\alpha,1_{a},\star,\varphi(a)) \to (\mylist{c},\mylist{\gamma},a,\alpha,b,\beta,f,\star,\Jelim{}(\varphi,a,b,f))
\end{align*}
in $\|\Gamma, x,y:A,z:\idn{A}{}(x,y),v:B(x,y,z)\|$ given by
\begin{align*}
  \expand{f}{\varphi} & =_{\textrm{def}} (1_{\mylist{c}},1_{a},f,\Jelim{}(\rr(\rr(x)),a,b,f),\Jelim{}(\rr(\varphi(x)),a,b,f)).
\end{align*}
Therefore, 
\begin{align*}
  \varphi[\mylist{\gamma},\alpha]\cdot(\expand{f}{\varphi})\forces_{\mylist{\gamma},\alpha,\beta,\star}\Jelim{}(\varphi,a,b,f):B(\mylist{c},a,b,f)
\end{align*}
and we define
\begin{align*}
  \Jelim{}(\varphi,a,b,f)[\mylist{\gamma},\alpha,\beta,\psi] & =_{\textrm{def}} \varphi[\mylist{\gamma},\alpha]\cdot(\expand{f}{\varphi}).
\end{align*}

It remains to be seen that this is functorial. Suppose given an arrow
\begin{align*}
  \theta=(\mylist{h},m,n,l):(\mylist{c},\mylist{\gamma},a,\alpha,b,\beta,f,\star)\to(\mylist{c}',\mylist{\gamma}',a',\alpha',b',\beta',f',\star).
\end{align*}
Then to show functoriality we need to prove that 
\begin{align}\label{eq:coherence_for_J}
  \Jelim{}(\varphi,a,b,f)[\mylist{\gamma},\alpha,\beta,\star]\cdot\bigl(\theta, \Jelim{}(\varphi,a',b',f')|(\mylist{h},m,n,l)\bigr)
  & =\Jelim{}(\varphi,a',b',f')[\mylist{\gamma}',\alpha',\beta',\star].
\end{align}
Well, we first observe that there is a commutative square
\begin{align}\label{eq:expand}
  \begin{minipage}{1.0\linewidth}
    \xy
    {\ar^{(\mylist{h},m,m',m'',\varphi|(\mylist{h},m,m',m''))}
      (0,20)*+{(\mylist{c},a,a,1_a,\varphi(a))};(75,20)*+{(\mylist{c'},a',a',1_{a'},\varphi(a'))}};
    {\ar_{\expand{f}{\varphi}}(0,20)*+{(\mylist{c},a,a,1_a,\varphi(a))};(0,0)*+{(\mylist{c},a,b,f,\Jelim{}(\varphi,a,b,f))}};
    {\ar^{\expand{f'}{\varphi}}(75,20)*+{(\mylist{c'},a',a',1_{a'},\varphi(a'))};(75,0)*+{(\mylist{c'},a',b',f',\Jelim{}(\varphi,a',b',f'))}};
    {\ar_{(\theta,\Jelim{}(\varphi,a,b,f)|\theta)}(0,0)*+{(\mylist{c},a,b,f,\Jelim{}(\varphi,a,b,f))};(75,0)*+{(\mylist{c'},a',b',f',\Jelim{}(\varphi,a',b',f'))}}
    \endxy
  \end{minipage}
\end{align}
where $m'$\/ is the composite, arising via weakening,
\begin{align*}
  \xy
  {\ar^{a \dagger
      (\mylist{h},m)}(0,0)*+{a\cdot(\mylist{h},m)};(30,0)*+{a\cdot\mylist{h}}};
  {\ar^{m}(30,0)*+{a\cdot\mylist{h}};(50,0)*+{a}};
  \endxy
\end{align*}
and similarly $m''$\/ is
\begin{align*}
  \xy
  {\ar^{1_a \dagger
      (\mylist{h},m,m')}(0,0)*+{1_{a}\cdot(\mylist{h},m,m')};(40,0)*+{1_{a}\cdot(\mylist{h},m)}};
  {\ar@{=}(40,0)*+{1_{a}\cdot(\mylist{h},m)};(60,0)*+{1_{a'}.}};
  \endxy
\end{align*}

Then, writing $\theta'=(\mylist{h},m,m',m'')$, we calculate

\begin{align*}
  \Jelim{}(\varphi,a,b,f)[\mylist{\gamma},\alpha,\beta,\star]\cdot
\bigl(\theta,\Jelim{}(\varphi,a,b,f)|\theta\bigr)
  & =
  \varphi[\mylist{\gamma},\alpha]\cdot(\expand{f}{\varphi})\cdot(\theta,\Jelim{}(\varphi,a,b,f)|\theta)\\
  & =
  \varphi[\mylist{\gamma},\alpha]\cdot(\theta',\varphi|\theta')\cdot(\expand{f'}{\varphi})\\
 & = \varphi[\mylist{\gamma},\alpha] \cdot (\mylist{h},m,\varphi|(\mylist{h},m))) \\
  & = \varphi[\mylist{\gamma}',\alpha']\cdot(\expand{f'}{\varphi})\\
  & = \Jelim{}(\varphi,a',b',f')[\mylist{\gamma}',\alpha',\beta',\star].
\end{align*}

Here, the first and last equalities are by definition of the realizers; the second equality is by virtue
of the fact that~\eqref{eq:expand} commutes; the third equality is by the definition of weakening,
and the fourth is by functoriality of the action on realizers.

\subsubsection{Conversion}

The conversion rule
\begin{prooftree}
  \AxiomC{$$}
  \UnaryInfC{$\judge{\Gamma,x:A}{\Jelim{}(\varphi,x,x,\rr(x)) = \varphi}$}
\end{prooftree}
is trivially seen to be satisfied using the definitions given above.

This concludes the proof of:

\begin{proposition}
  The CwF $\real{R}$\/ supports identity types.
\end{proposition}

\subsection{Semantic dependent sums}

Consider a context $\lscott \Gamma\rscott$, a type $A$\/ over this context with comprehension
$\lscott \Gamma,x:A\rscott$\/ and furthermore a type $B$\/ over this comprehension, giving
a further comprehension $\lscott \Gamma,x:A,y:B(x)\rscott$. Thus we have a diagram
\[
\xymatrix{
\lscott \Gamma,x:A,y:B(x)\rscott \ar[r] \ar[d] & \lscott \Gamma,x:A\rscott \ar[d] \ar[r] 
& \lscott \Gamma\rscott \ar[d] \\
|\Gamma,x:A,y:B(x)| \ar[r] & |\Gamma,x:A| \ar[r] & |\Gamma|.
}
\]
We wish to define the semantic dependent sum type in context $\lscott \Gamma\rscott$.
To this end, we first form the syntactic type $\Gamma \vdash \sum_{x:A}B(x)$,
giving rise to the pullback diagram
\[
\xymatrix{
\|\Gamma,v:\sum_{x:A}B(x)\| \ar[r] \ar[d] & |\Gamma,v:\sum_{x:A}B(x)| \ar[d] \\
\lscott \Gamma \rscott \ar[r] & |\Gamma|.
}
\]
We then must specify the realizability presheaf 
\[ \|\Gamma,v:\sum_{x:A}B(x)\| \to \sets.\]
For $(\mylist{c},\mylist{\gamma})$\/ an object of
$\lscott\Gamma\rscott$, we define
\begin{align*}
  \upsilon\forces_{\mylist{\gamma}}p:\sum_{x:A(\mylist{c})}B(\mylist{c},x)
\end{align*}
if and only if $\upsilon$ is a pair $(\upsilon_{0},\upsilon_{1})$ such
that
\begin{align*}
  \upsilon_{0}&\forces_{\mylist{\gamma}}\pi_{0}(p):A(\mylist{c}),\text{
    and}\\
  \upsilon_{1}&\forces_{\mylist{\gamma},\upsilon_{0}}\pi_{1}(p):B(\mylist{c},\pi_{0}(p)).
\end{align*}

\subsubsection{Reindexing}

Suppose we are given a map
$(\mylist{h},m):(\mylist{c},\gamma,p)\to(\mylist{c}',\gamma',p')$ in the groupoid 
$\|\judge{\Gamma}{\sum_{x:A}B(x)}\|$ together with a realizer
$\upsilon\forces_{\mylist{\gamma}}p:\sum_{x:A(\mylist{c})}B(\mylist{c},x)$.
Then we have
\begin{align*}
  \upsilon_{0}\cdot(\mylist{h},\xi_{0})\forces_{\mylist{\gamma}'}\pi_{0}(p'):A(\mylist{c}')
\end{align*}
where $\xi_{0}$ is the composite
\begin{align*}
  \xy
  {\ar^{\Jelim{}^{\sigma}(\rr\pi_{0}(v),\mylist{c},\mylist{c}',\mylist{h},p)}(0,0)*+{\pi_{0}(p)\cdot\mylist{h}};(50,0)*+{\pi_{0}(p\cdot\mylist{h})}};
  {\ar^{\Jelim{}(\rr\pi_{0}(x),p\cdot\mylist{h},p',m)}(50,0)*+{\pi_{0}(p\cdot\mylist{h})};(100,0)*+{\pi_{0}(p')}};
  \endxy
\end{align*}
Similarly, we define $\xi_{1}$ to be the composite
\[
\xymatrix{
\pi_{1}(p)\cdot(\mylist{h},\xi_0) \ar[r]^-{\xi'} &
\pi_{1}(p\cdot\mylist{h})\cdot(1_{\mylist{c}'},\Jelim{}(\rr\pi_{0}x,p\cdot \mylist{h},p',m))
\ar[rr]^-{\Jelim{}(\rr\pi_{0}x,p\cdot\mylist{h},p',m))}&&
\pi_{1}(p')
}
\]
where $\xi'=\Jelim{}^{\sigma}(\rr\pi_{1}(v)\cdot(1_{\mylist{x}},\Jelim{}(\rr\pi_{0}x,v,w,u)),\mylist{c},\mylist{c}',p)$.

Then we have that 
\begin{align*}
  \upsilon_{1}\cdot(\mylist{h},\xi_{0},\xi_{1})\forces_{\mylist{\gamma}',\upsilon_{0}\cdot(\mylist{h},\xi_{0})}\pi_{1}(p'):B(\mylist{c}',\pi_{0}(p'))
\end{align*}
and we define
\begin{align*}
  \upsilon\cdot(\mylist{h},m) & =_{\textrm{def}} \bigl(\upsilon_{0}\cdot(\mylist{h},\xi_{0}),\upsilon_{1}\cdot(\mylist{h},\xi_{0},\xi_{1})\bigr).
\end{align*}
Functoriality of reindexing is a routine verification using the truncation rule.

Next, we must show that the formation of the dependent sum type is stable under substitution. 
It suffices to show that, for a context morphism $m:\Delta \to \Gamma$, the diagram
\[
\xymatrix{
\|\Delta,v:\sum_{A[m]}B[m]\| \ar[r] \ar[dr]_{\forces} & \|\Gamma,v:\sum_{A}B\| \ar[d]^\forces\\
& \sets
}
\]
is commutative. To this end, consider an object $(\mylist{d},\mylist{\delta},p)$\/ and note that we have
\begin{eqnarray*}
\upsilon \forces_{m[\mylist{\delta}]} p: \sum_{x:A[m]}B[m]  
& \Leftrightarrow & \upsilon_0 \forces_{m[\mylist{\delta}]} \pi_0(p):A[m] \text{ and }
\upsilon_1 \forces_{m[\mylist{\delta}], \upsilon_0} B[m](\pi_0(p)) \\
& \Leftrightarrow & \upsilon_0 \forces_{\mylist{\delta}} \pi_0(p):A \text{ and }
\upsilon_1 \forces_{\mylist{\delta},\upsilon_0} \pi_1(p):B(\pi_0(p))
\end{eqnarray*}
where we have used Lemma~\ref{lem:real_subst} regarding the behavior of realizers under substitution.
\subsubsection{Introduction rule}

For the introduction rule 
\begin{prooftree}
  \AxiomC{$\judge{\Gamma}{a:A}$}
  \AxiomC{$\judge{\Gamma}{b:B(a)}$}
  \BinaryInfC{$\judge{\Gamma}{\pair(a,b):\sum_{x:A}B(x)}$}
\end{prooftree}
we reason as follows.  For an object $(\mylist{c},\mylist{\gamma})$ of
$\lscott\Gamma\rscott$ we define
\begin{align*}
  \pair(a,b)[\mylist{\gamma}] & =_{\textrm{def}} (a[\mylist{\gamma}],b[\mylist{\gamma}]).
\end{align*}
For the coherence property note that, given
$\mylist{h}:(\mylist{c},\mylist{\gamma})\to(\mylist{c}',\mylist{\gamma}')$, we have
\begin{align*}
  \pair(a,b)[\mylist{\gamma}]\cdot(\mylist{h},\pair(a,b)|\mylist{h}) &
  =
  (a[\mylist{\gamma}],b[\mylist{\gamma}])\cdot(\mylist{h},\pair(a,b)|\mylist{h})\\
  & = \bigl(a[\mylist{\gamma}]\cdot(\mylist{h},\xi_{0}),b[\mylist{\gamma}]\cdot(\mylist{h},\xi_{0},\xi_{1})\bigr).
\end{align*}
Since, in this case, the truncation rule gives that
$\xi_{0}=a|\mylist{h}$ and $\xi_{1}=b|(\mylist{h},\xi_{0})$, the
required equation is then immediate, and it also follows without difficulty that this term formation
operation commutes with substitution with respect to context morphisms.

\subsubsection{Elimination rule}

For the elimination rule
\begin{prooftree}
  \AxiomC{$\judge{\Gamma,z:\sum_{x:A}B(x)}{C(z)}$}
  \AxiomC{$\judge{\Gamma,x:A,y:B(x)}{\psi:C(\pair(x,y))}$}
  \BinaryInfC{$\judge{\Gamma,z:\sum_{x:A}B(x)}{\RSigma(\psi,z):C(z)}$}
\end{prooftree}
we assume given $(\mylist{c},\mylist{\gamma},p,\upsilon)$ in
$\lscott\Gamma,z:\sum_{x:A}B(x)\rscott$.  Then 
\begin{align*}
  \psi[\mylist{\gamma},\upsilon_{0},\upsilon_{1}]\forces_{\mylist{\gamma},\upsilon_{0},\upsilon_{1}}\psi(\mylist{c},\pi_{0}p,\pi_{1}p):C(\mylist{c},\pair(\pi_{0}p,\pi_{1}p)),
\end{align*}
and so
\begin{align*}
  \psi[\mylist{\gamma},\upsilon_{0},\upsilon_{1}]\forces_{\mylist{\gamma},\upsilon}\psi(\mylist{c},\pi_{0}p,\pi_{1}p):C(\mylist{c},\pair(\pi_{0}p,\pi_{1}p)).
\end{align*}
Now, we have the term
\begin{prooftree}
  \AxiomC{$\judge{\Gamma,z:\sum_{x:A}B(x)}{\idn{\sum_{x:A}B(x)}{}\bigl(\pair(\pi_{0}z,\pi_{1}z),z\bigr)}$}
  \noLine
  \UnaryInfC{$\judge{\Gamma,x:A,y:B(x)}{\rr(\pair(x,y)):\idn{\sum_{x:A}B(x)}{}\bigl(\pair(x,y),\pair(x,y)\bigr)}$}
  \UnaryInfC{$\judge{}{\RSigma(\rr(\pair(x,y)),p):\idn{\sum_{x:A(\mylist{c})}B(\mylist{c},x)}{}\bigl(\pair(\pi_{0}p,\pi_{1}p),p\bigr)}$}
\end{prooftree}
Now, it is easily seen, by inspecting the definition of the reindexing
action on realizers for dependent sums, that in fact we have an arrow
\begin{align*}
  (1_{\mylist{c}},\RSigma(\rr(\pair(x,y),p)):(\mylist{c},\mylist{\gamma},\pair(\pi_{0}p,\pi_{1}p),\upsilon)\to(\mylist{c},\mylist{\gamma},p,\upsilon)
\end{align*}
in
$\lscott\judge{\Gamma}{\sum_{x:A}B(x)}\rscott_{(\mylist{c},\mylist{\gamma},p)}$.
Therefore,
\begin{align*}
  \psi[\mylist{\gamma},\upsilon_{0},\upsilon_{1}]\cdot(1_{\mylist{c}},\RSigma(\rr\pair(x,y),p),\RSigma(\rr\psi(\mylist{c},\pi_{0}z,\pi_{1}z),p)\forces_{\mylist{\gamma},\upsilon}\RSigma(\psi,p):C(\mylist{c},p).
\end{align*}
Here the term $\RSigma(\rr\psi(\mylist{c},x,y),p)$ is constructed
as follows:
\begin{prooftree}
  \AxiomC{$\judge{z:\sum_{x:A(\mylist{c})}B(\mylist{c},x)}{\idn{C(\mylist{c},z)}{}\bigl(\psi(\mylist{c},\pi_{0}z,\pi_{1}z)\cdot(1_{\mylist{c}},\RSigma(\rr(\pair(x,y)),z)),\RSigma(\psi(\mylist{c},x,y),z)\bigr)}$}
  \noLine
  \UnaryInfC{$\judge{x:A,y:B(x)}{\rr\psi(\mylist{c},x,y):\idn{C(\mylist{c},\pair(x,y))}{}\bigl(\psi(\mylist{c},x,y),\psi(\mylist{c},x,y)\bigr)}$}
  \UnaryInfC{$\judge{}{\RSigma(\rr\psi(\mylist{c},x,y),p):\idn{C(\mylist{c},p)}{}\bigl(\psi(\mylist{c},\pi_{0}p,\pi_{1}p)\cdot(1_{\mylist{c}},\RSigma(\rr(\pair(x,y)),p)),\RSigma(\psi(\mylist{c},x,y),p)\bigr)}$.}
\end{prooftree}
So we define
\begin{align*}
  \RSigma(\psi,p)[\mylist{\gamma},\upsilon] & =_{\textrm{def}} \psi[\mylist{\gamma},\upsilon_{0},\upsilon_{1}]\cdot(1_{\mylist{c}},\RSigma(\rr\pair(x,y),p),\RSigma(\rr\psi(\mylist{c},\pi_{0}z,\pi_{1}z),p)
\end{align*}

\subsubsection{Conversion rule}

The conversion rule
\begin{prooftree}
  \AxiomC{$$}
  \UnaryInfC{$\judge{\Gamma,x:A,y:B(x)}{\RSigma(\psi,\pair(x,y)) \;=\;\psi(x,y):C(\pair(x,y))}$}
\end{prooftree}
is trivial given the definitions above.

This concludes the proof of:
\begin{proposition}
  The CwF $\real{R}$\/ supports dependent sums.
\end{proposition}

\subsection{Semantic natural numbers}

Although it is in principle possible to consider different interpretations of
the type of natural numbers (when suitable notions of realizability
are provided) we will simply fix a uniform interpretation of natural
numbers.  Observe that, in the groupoid model built over the free
groupoid on $G$ the type of natural numbers $\nat$ is interpreted as
the discrete groupoid of natural numbers.  In particular, in this
interpretation, each $t:\nat$ is sent to a genuine (external) natural
number $n_{t}$.  This $n_{t}$ can be represented by theoretically by
the corresponding number, which we denote by $\#(t)$.  Thus, for
$t:\nat$, we have a numeral $\#(t):\nat$.  By the construction of the
groupoid model, this gives an endofunctor $\#:|\nat|\to|\nat|$ and we define
\begin{align*}
  \tau\forces t :\nat \quad\text{ if and only if
  }\tau:\idn{\nat}{}(t,\# t).
\end{align*}
Note that for any $f:\idn{\nat}{}(s,t)$ we have $\# f=\rr(\# s)$ since
$\nat$ is interpreted as a discrete groupoid under the groupoid
interpretation.

By the functoriality of $\#$ this determines the interpretation
$\lscott\nat\rscott$ of the type of natural numbers in $\real{R}$.
Observe that $\#(n)=n$ when $n$ is a numeral.  The successor operation
gives an endofunctor $|\successor|:|\nat|\to|\nat|$ and, by inspection
of the groupoid interpretation, this functor
commutes with $\#$ in the sense that
$\#\circ|\successor|=|\successor|\circ\#$.  Using these facts it is
possible to describe the interpretations of the terms coming from the
introduction rules.  First, we take $\lscott \zero\rscott$ to be given by the realizer
$\rr(\zero)\forces\zero:\nat$.  Next, given $\tau\forces t:\nat$, we
have, by the aforementioned facts, that $|S|(\tau)\forces S(t):\nat$.
As such, we take this term to be the realizer part of the interpretation of
$\lscott\successor(t)\rscott$.  The interpretation of the elimination
rule will require a little bit more care.

\subsubsection{Elimination rule}

We now will show that the elimination rule
\begin{prooftree}
  \AxiomC{$\judge{\Gamma}{b:B(\zero)}$}
  \AxiomC{$\judge{\Gamma,x:\nat,y:B(x)}{g(x,y):B\bigl(\successor(x)\bigr)}$}
  \BinaryInfC{$\judge{\Gamma,n:\nat}{\rec(b,g,n):B(n)}$}
\end{prooftree}
is valid in the interpretation. 

Let $(\mylist{c},\mylist{\gamma},n,\nu)$ in $\lscott\Gamma,n:\nat\rscott$ be
given.  In what follows we will suppress $\mylist{c}$ as much as
possible in order to avoid notational clutter.  

We will construct the interpretation of the elimination term by induction on
$\# n$.  In the base case we take
$\rec(b,g)[\mylist{\gamma},\nu]\forces_{\mylist{\gamma},\nu}\rec(b,g,n):B(\mylist{c},n)$
to be given by
\begin{align*}
  b[\mylist{\gamma}]\cdot\bigl(\rec(b,g)|(1_{\mylist{c}},\nu^{-1})\bigr)\forces_{\mylist{\gamma},\nu}\rec(n,b,g):B(\mylist{c},n).
\end{align*}

When $\# n=\successor(\mathtt{m})$ for $\mathtt{m}$ a
numeral, we have by induction hypothesis, 
\begin{align*}
  \rec(b,g)[\mylist{\gamma},\rr(\mathtt{m})]\forces_{\mylist{\gamma},\rr(\mathtt{m})}\rec(b,g,\mathtt{m}):B(\mylist{c},\mathtt{m}).
\end{align*}
By the conversion rule
$\rec(b,g,\successor(\mathtt{m}))=g(\mathtt{m},\rec(b,g,\mathtt{m}))$
and therefore
\begin{align*}
  g\bigl[\mylist{\gamma},\rr(\mathtt{m}),\rec(b,g)[\mylist{\gamma},\rr(\mathtt{m})]\bigr]\forces_{\mylist{\gamma},\rr(\mathtt{m}),\rec(b,g)[\mylist{\gamma},\rr(\mathtt{m})]}\rec(b,g,\successor(\mathtt{m})):B(\mylist{c},\successor(\mathtt{m})),
\end{align*}
which holds if and only if 
\begin{align*}
  g\bigl[\mylist{\gamma},\rr(\mathtt{m}),\rec(b,g)[\mylist{\gamma},\rr(\mathtt{m})]\bigr]\forces_{\mylist{\gamma},\rr(\successor(\mathtt{m}))}\rec(b,g,\successor(\mathtt{m})):B(\mylist{c},\successor(\mathtt{m})).
\end{align*}
As such, we may define 
\begin{align*}
  \rec(b,g,n)[\mylist{\gamma}] \; =_{\textrm{def}}\; g\bigl[\mylist{\gamma},\rr(\mathtt{m}),\rec(b,g)[\mylist{\gamma},\rr(\mathtt{m})]\bigr]\cdot\bigl(\rec(b,g)|(1_{\mylist{c}},\nu^{-1})\bigr)\forces_{\mylist{\gamma},\nu}\rec(n,b,g):B(\mylist{c},n).
\end{align*}

\subsubsection{Conversion rule}

The conversion rule $\rec(b,g,\zero)=b$ is immediate.  For the conversion rule
\begin{align*}
  \judge{\Gamma,n:\nat}{\rec\bigl(b,g,\successor(n)\bigr) = g\bigl(n,\rec\bigl(b,g,n\bigr)\bigr) :B\bigl(\successor(n)\bigr)}
\end{align*}
we proceed by induction on $\# n$.  In the base case, where $\#
n=\zero$, we have 
\begin{align*}
  \rec(b,g)[\mylist{\gamma},|S|(\nu)] & =
  g\bigl[\mylist{\gamma},\rr(\zero),b[\mylist{\gamma}]\bigr]\cdot\bigl(\rec(b,g)|(1_{\mylist{c}},|S|(\nu)^{-1})\bigr)\\
  & =
  g\bigl[\mylist{\gamma},\rr(\zero),b[\mylist{\gamma}]\bigr]\cdot\bigl(g|(1_{\mylist{c}},\nu^{-1},\rec(b,g)|(1_{\mylist{c}},\nu^{-1}))\bigr)\\
  & =
  g\bigl[\mylist{\gamma},\nu,b[\mylist{\gamma}]\cdot(\rec(b,g)|(1_{c},\nu^{-1}))\bigr]\\
  & = g\bigl[\mylist{\gamma},\nu,\rec(b,g)[\mylist{\gamma},\nu]\bigr],
\end{align*}
where the penultimate equation is by functoriality of the
interpretation of $g$ and the equation 
\begin{align*}
  \rec(b,g)|(1_{\mylist{c}},|S|(\nu)^{-1}) & = g|(1_{\mylist{c}},\nu^{-1},\rec(b,g)|(1_{\mylist{c}},\nu^{-1}))
\end{align*}
is by 1-truncation.  In the case where $\# n=\successor(\mathtt{m})$ for
$\mathtt{m}$ the equation follows from the induction hypothesis and
the same reasoning as in the base case.

This completes the proof of:

\begin{proposition}
  The CwF $\real{R}$\/ supports the type of natural numbers.
\end{proposition}

\appendix
\begin{landscape}

\section{Rules of type theory}\label{app:rules}

\subsubsection*{Structural rules}

~
{\small
\begin{center}
  \AxiomC{$\judge{\Gamma}{\mathcal{J}}$}
  \RightLabel{($\dag$)}
  \UnaryInfC{$\judge{\Delta,\Gamma}{\mathcal{J}}$}
  \DisplayProof
  \hspace{.5cm}
   \AxiomC{$a:A$}
  \AxiomC{$\judge{x:A,\Delta}{B(x)}$}
   \BinaryInfC{$\judge{\Delta[a/x]}{B(a)}$}
  \DisplayProof
  \hspace{.5cm}
  \AxiomC{$a:A$}
  \AxiomC{$\judge{x:A,\Delta}{b(x):B(x)}{}$}
  \BinaryInfC{$\judge{\Delta[a/x]}{b(a):B(a)}$}
  \DisplayProof
  \hspace{.5cm}
  \AxiomC{$A$}
  \UnaryInfC{$\judge{x:A,\Delta}{x:A}$}
  \DisplayProof
\end{center}}
($\dag$): $\mathcal{J}$ ranges over judgements and the
variables declared in $\Delta$ and $\Gamma$ are disjoint.

\subsubsection*{Equality rules}

\begin{center}
{\small
  \AxiomC{$A$}
  \UnaryInfC{$A=A$}
  \DisplayProof
  \hspace{.5cm}
  \AxiomC{$A=B$}
  \UnaryInfC{$B=A$}
  \DisplayProof
  \hspace{.5cm}
  \AxiomC{$A=B$}
  \AxiomC{$B=C$}
  \BinaryInfC{$A=C$}
  \DisplayProof
  \hspace{.5cm}
  \AxiomC{$a:A$}
  \UnaryInfC{$a=a:A$}
  \DisplayProof
  \hspace{.5cm}
  \AxiomC{$a=b:A$}
  \UnaryInfC{$b=a:A$}
  \DisplayProof
  \hspace{.5cm}
  \AxiomC{$a=b:A$}
  \AxiomC{$b=c:A$}
  \BinaryInfC{$a=c:A$}
  \DisplayProof
}
\end{center}
\vspace{.25cm}
\begin{center}
  {\small
    \AxiomC{$a=b:A$}
    \AxiomC{$\judge{x:A}{B(x)}$}
  \BinaryInfC{$B(a)=B(b)$}
  \DisplayProof
  \hspace{1cm}
  \AxiomC{$a=b:A$}
  \AxiomC{$\judge{x:A}{f(x):B(x)}$}
  \BinaryInfC{$f(a)=f(b):B(a)$}
  \DisplayProof
  \hspace{1cm}
  \AxiomC{$A=B$}
  \AxiomC{$a:A$}
  \BinaryInfC{$a:B$}
  \DisplayProof}
\end{center}

\subsubsection*{Formation rules}

\begin{center}
{\small
  \AxiomC{$\judge{x:A}{B(x)}$}
  \UnaryInfC{$\prod_{x:A}B(x)$}
  \DisplayProof
  \hspace{1cm}
  \AxiomC{$\judge{x:A}{B(x)}$}
  \UnaryInfC{$\sum_{x:A}B(x)$}
  \DisplayProof
  \hspace{1cm}
  \AxiomC{$a,b:A$}
  \UnaryInfC{$\judge{}{\idn{A}{}(a,b)}$}
  \DisplayProof
  \hspace{1cm}
  \AxiomC{}
  \UnaryInfC{$\judge{}{\nat}$}
  \DisplayProof}
\end{center}

\subsubsection*{Introduction rules}

\begin{center}
{\small
  \AxiomC{$\judge{x:A}{f(x):B(x)}$}
  \UnaryInfC{$\lambda_{x:A}f(x):\prod_{x:A}B(x)$}
  \DisplayProof
  \hspace{.5cm}
  \AxiomC{$a:A$}
  \AxiomC{$b:B(a)$}
  \BinaryInfC{$\pair(a,b):\sum_{x:A}B(x)$}
  \DisplayProof
  \hspace{.5cm}
  \AxiomC{$a:A$}
  \UnaryInfC{$r(a):\idn{A}{}(a,a)$}
  \DisplayProof
  \hspace{.5cm}
  \AxiomC{}
  \UnaryInfC{$\zero:\nat$}
  \DisplayProof
  \hspace{.5cm}
  \AxiomC{$n:\nat$}
  \UnaryInfC{$\successor(n):\nat$}
  \DisplayProof}
\end{center}

\subsubsection*{Elimination rules}

\begin{center}
{\small
  \AxiomC{$f:\prod_{x:A}B(x)$}
  \AxiomC{$a:A$}
  \BinaryInfC{$\app(f,a):B(a)$}
  \DisplayProof
  \hspace{1cm}
  \AxiomC{$\judge{}{p:\sum_{x:A}B(x)}$}
  \AxiomC{$\judge{x:A,y:B(x)}{\psi(x,y):C\bigl(\pair(x,y)\bigr)}$}
  \BinaryInfC{$\RSigma\bigl([x:A,y:B(x)]\psi(x,y),p\bigr):C(p)$}
  \DisplayProof}
\end{center}
\vspace{.25cm}
\begin{center}
{\small
  \AxiomC{$\judge{x:A,y:A,z:\idn{A}{}(x,y)}{B(x,y,z)}$}
  \noLine
  \UnaryInfC{$\judge{x:A}{\varphi(x):B\bigl(x,x,r(x)\bigr)}$}
  \noLine
  \UnaryInfC{$f:\idn{A}{}(a,b)$}
  \UnaryInfC{$\Jelim{}\bigl(\varphi,a,b,f):B(a,b,f)$}
  \DisplayProof
  \hspace{1.5cm}
  \AxiomC{$n:\nat$}
  \AxiomC{$b:B(\zero)$}
  \AxiomC{$\judge{x:\nat,y:B(x)}{g(x,y):B\bigl(\successor(x)\bigr)}$}
  \TrinaryInfC{$\rec\bigl(b,g,n\bigr):B(n)$}
  \DisplayProof}
\end{center}

\subsubsection*{Conversion rules}

\begin{center}
  {\small
  \AxiomC{$\lambda_{x:A}f(x):\prod_{x:A}B(x)$}
  \AxiomC{$a:A$}
  \BinaryInfC{$\app\bigl(\lambda_{x:A}f(x),a\bigr) \;=\; f(a):B(a)$}
  \DisplayProof
  \hspace{.5cm}
  \AxiomC{$a:A$}
  \AxiomC{$b:B(a)$}
  \AxiomC{$\judge{x:A,y:B(x)}{\psi(x,y):C\bigl(\pair(x,y)\bigr)}$}
  \TrinaryInfC{$\RSigma\bigl(\psi,\pair(a,b)\bigr) \;=\; \psi(a,b):C\bigl(\pair(a,b)\bigr)$}
  \DisplayProof
}
\end{center}
\vspace{.25cm}
\begin{center}
  {\small
  \AxiomC{$a:A$}
  \UnaryInfC{$\Jelim{}\bigl(\varphi,a,a,r(a)\bigr)\;=\;\varphi(a):B\bigl(a,a,r(a)\bigr)$}
  \DisplayProof
  \hspace{.5cm}
  \AxiomC{}
  \UnaryInfC{$\rec\bigl(b,g,\zero\bigr) = b :B(\zero)$}
  \DisplayProof
  \hspace{.5cm}
  \AxiomC{$n:\nat$}
  \UnaryInfC{$\rec\bigl(b,g,\successor(n)\bigr) = g\bigl(n,\rec\bigl(b,g,n\bigr)\bigr) :B\bigl(\successor(n)\bigr)$}}
\end{center}
\end{landscape}

\section{Syntactic constructions}\label{app:synt}

\subsection{Parameterized $\Jelim{}$-terms}\label{app:par}

We describe a parameterized version of the elimination rule for identity types. Consider a context
$\Delta$\/ of the form
\[ \Delta = \bigl( x,y:A,z:\idn{A}{}(x,y), v_1:B_1(x,y,z), \ldots, v_n:B_n(x,y,z,\vec{v})\bigr).\]

Then the following rule gives the parameterized terms $\Jelim{}(\varphi,a,b,f,\vec{v})$:

\begin{prooftree}
  \AxiomC{$\judge{\Delta}{T(x,y,z,\vec{v})}$}
  \noLine
  \UnaryInfC{$\judge{\Delta[x/y,\rr(x)/z]}{\varphi(x,\vec{v}):T\bigl(x,x,r(x),\vec{v}\bigr)}$}
  \noLine
  \UnaryInfC{$f:\idn{A}{}(a,b)$}
  \RightLabel{$\id{}$ elimination}
  \UnaryInfC{$\judge{\Delta[a/x,b/y,f/z]}{\Jelim{}([x:A]\varphi,a,b,f,\vec{v}):T(a,b,f,\vec{v})}$}
\end{prooftree}

The construction is by induction on the length of the list of parameters $\vec{v}$. When $n=0$\/
we simply set $\Jelim{}(\varphi,a,b,f,())=\Jelim{}(\varphi,a,b,f)$. Next, assume we have constructed the
terms
$\Jelim{}(\varphi,a,b,f,v_1,\ldots,v_m)$\/ for lists of parameters of length $m<n$. 
Assume we are given a judgement of the form
\[ \judge{\Delta[x/y,\rr(x)/z]}{\varphi(x,v_1,\ldots,v_n):T(x,x,\rr(x),v_1,\ldots,v_n)}. \]
We may then form the term
\[ \lambda_{v_n:B_n(x,x,\rr(x),\vec{v})}\varphi(x,\vec{v}): 
\prod_{v_n:B_n(x,x,\rr(x),\vec{v})} T(x,x,\rr(x),\vec{v}) \]
in context
\[ \bigl( x:A, v_1:B_1(x,x,\rr(x)),\ldots, v_{n-1}:B_{n-1}(x,x,\rr(x),\vec{v})\bigr). \]
Thus by induction hypothesis we have the parameterized term

\begin{prooftree}
  \AxiomC{$\judge{x,y:A,z:\idn{A}{}(x,y), v_1:B_1(x,y,z), \ldots, v_{n-1}:B_{n-1}(x,y,z,\vec{v})}{\prod_{v_n}T(x,y,z,\vec{v})}$}
  \noLine
  \UnaryInfC{$\judge{\bigl( x:A, v_1:B_1(x,x,\rr(x)),\ldots, v_{n-1}:B_{n-1}(x,x,\rr(x),\vec{v})\bigr)}{\lambda_{v_n}\varphi(x,\vec{v}):\prod_{v_n}T(x,x,r(x),\vec{v})}$}
  \noLine
  \UnaryInfC{$f:\idn{A}{}(a,b)$}
    \UnaryInfC{$\judge{v_1:B_1(a,b,f), \ldots, v_{n-1}:B_{n-1}(a,b,f,\vec{v})}
{\Jelim{}(\lambda_{v_n}\varphi(x,\vec{v}),a,b,f,\vec{v}):\prod_{v_n}T(a,b,f,\vec{v})}$}
\end{prooftree}
Hence we may apply this term to $v_n$\/ to obtain the desired term:
\[ \Jelim{}(\varphi,a,b,f,\vec{v}) =_{\textrm{def}} 
\app\bigl(\Jelim{}(\lambda_{v_n}\varphi(x,\vec{v}),a,b,f,\vec{v}),v_n\bigr).\]

It is readily verified that these terms satisfy the conversion rule

\begin{prooftree}
 \AxiomC{$\judge{}{a:A}$}
\UnaryInfC{$\judge{\Delta[a/x,a/y,\rr(a)/z]}{\Jelim{}(\varphi,a,a,\rr(a),\vec{v}) = \varphi(a,\vec{v}):
T(a,a,\rr(a),\vec{v})}  $}
\end{prooftree}

\subsection{Sequential $\Jelim{}$-terms}\label{app:seq}

We now turn to the construction of sequential $\Jelim{}$-terms, which arise in the setting of 
repeated Grothendieck constructions. We will actually need a parameterized version of these 
as well. That is, given a context 
\[ \Gamma=\bigl(x_1:A_1, \ldots x_n:A_n(x_1, \ldots, x_{n-1})\bigr)\]
we consider an extended context
\[ \Delta = \bigl(x_1:A_1, \ldots x_n:A_n(x_1, \ldots, x_{n-1}), \vec{v} \bigr)\]
where the variables $v_n$\/ are considered as parameters. Then we have the associated context
\[ \tilde{\Delta} =_{\textrm{def}} \tilde{\Gamma}, \vec{v}. \]
We wish to establish the following derived rule:

\begin{prooftree}
\AxiomC{$\judge{\tilde{\Delta}}{T}$}
\noLine
\UnaryInfC{$\judge{\Delta}{\varphi:T[\mylist{x}/\mylist{y},\rr(\mylist{x})/\mylist{z},\vec{v}]}  $}
\UnaryInfC{$\judge{\tilde{\Delta}}{\Jelim{}^\sigma([\mylist{x}]\varphi,\mylist{x},\mylist{y},
\mylist{z},\vec{v}):T(\mylist{x},\mylist{y},\mylist{z},\vec{v}) }  $}
\end{prooftree}

This is done by induction on the length of the context $\Gamma$. When $\Gamma=(x_1:A_1)$,
the problem reduces to the construction of the parameterized term
$\Jelim{}([x]\varphi,x,y,z,\vec{v})$, which was done in the previous section.

Thus assume that we have constructed the parameterized sequential terms for contexts
of length $n-1$. First consider the context

\begin{eqnarray*}
\Delta' & =_{\textrm{def}} & \Delta[x_1/y_1, \rr(x_1)/z_1, \ldots, x_{n-1}/y_{n-1}, \rr(x_{n-1})/z_{n-1}] \\
 & = & \bigl(x_1:A_1,\ldots, x_{n-1}:A_{n-1}, x_n,y_n:A_n, z_n:\idn{A_n}{}(x_n,y_n), \vec{v} 
\bigr)
\end{eqnarray*} 
Then we have
\[ \judge{\Delta'}{T(x_1,x_1,\rr(x_1), \ldots, x_{n-1},x_{n-1},\rr(x_{n-1}),x_n,y_n,z_n,\vec{v})}. \]
Noting that $\Delta'[x_n/y_n,\rr(x_n)/z_n]=\Delta$, we may thus use the ordinary parameterized
elimination rule to form the judgement

\begin{prooftree}
\AxiomC{$\judge{\Delta'}{T(x_1,x_1,\rr(x_1), \ldots, x_{n-1},x_{n-1},\rr(x_{n-1}),x_n,y_n,z_n,\vec{v})}$}
\noLine
\UnaryInfC{$\judge{\Delta}{\varphi:T[\mylist{x}/\mylist{y},\rr(\mylist{x})/\mylist{z},\vec{v}]}  $}
\UnaryInfC{$\judge{\Delta'}{\Jelim{}([x_n]\varphi,x_n,y_n,z_n,\vec{v}):
T((x_1,x_1,\rr(x_1), \ldots, x_{n-1},x_{n-1},\rr(x_{n-1}),x_n,y_n,z_n,\vec{v})) }  $}
\end{prooftree}

We can now apply the induction hypothesis to this term (where the parameters are now 
$x_n,y_n,z_n,\vec{v}$).

\section{Realizability clauses}\label{app:real}

Here we collect for easy reference the realizability clauses extracted
from the interpretation of type theory described above.

\subsection{Structural rules}

\subsubsection*{Weakening}

Assume we have judgements $\judge{\Gamma}{A}$ and
$\judge{\Delta}{}$.  In terms of realizers, given
$(\mylist{c},\mylist{\gamma},\mylist{d},\mylist{\delta})$ and
object of $\lscott\Gamma,\Delta\rscott$ and a term $a:A(\mylist{c})$ we have
\begin{align*}
  \alpha\forces_{\mylist{\gamma},\mylist{\delta}}a:A(\mylist{c})
  \quad\text{ iff
  }\quad\alpha\forces_{\mylist{\gamma}}a:A(\mylist{c}).
\end{align*}
For reindexing, given an arrow
$(\mylist{h},\mylist{k},m):(\mylist{c},\mylist{\gamma},\mylist{d},\mylist{\delta},a)\to(\mylist{c}',\mylist{\gamma}',\mylist{d}',\mylist{\delta}',a')$
in $\|\judge{\Gamma,\Delta}{A}\|$, we have\footnote{Explicitly, we
  have, in the notation of Section \ref{sec:synt}, that 
  \begin{align*}
    \alpha\cdot(\mylist{h},\mylist{k},m) &
    = \alpha\cdot\bigl(\mylist{h},m\circ(a\dag(\mylist{h},\mylist{k}))\bigr).
  \end{align*}
}
\begin{align*}
  \alpha\cdot(\mylist{h},\mylist{k},m) &
  = \alpha\cdot\hat{\tau}_{\Gamma;\Delta;A}(\mylist{h},\mylist{k},m).
\end{align*}
For terms, given $\judge{\Gamma}{a:A}$, we have 
\begin{align*}
  a[\mylist{\gamma},\mylist{\delta}] & = a[\mylist{\gamma}].
\end{align*}

\subsubsection*{Substitution}

Given 
\begin{prooftree}
  \AxiomC{$\judge{\Gamma}{a:A}$}
  \AxiomC{$\judge{\Gamma,x:A,\Delta}{B(x)}$}
  \BinaryInfC{$\judge{\Gamma,\Delta[a/x]}{B(a)}$}
\end{prooftree}
we have 
\begin{align*}
  \beta\forces_{\mylist{\gamma},\mylist{\delta}}b:B(\mylist{c},a(\mylist{c}),\mylist{d})
  \quad \text{ iff }\quad\beta\forces_{\mylist{\gamma},a[\mylist{\gamma}],\mylist{\delta}}b:B(\mylist{c},a(\mylist{c}),\mylist{d})
\end{align*}
and 
\begin{align*}
  \beta\cdot(\mylist{h},\mylist{k},m) &
  = \beta\cdot\hat{\sigma}_{\Gamma;a;\Delta}(\mylist{h},\mylist{k},m)
\end{align*}
for $(\mylist{h},\mylist{k},m):(\mylist{c},\mylist{\gamma},\mylist{d},\mylist{\delta},b)\to(\mylist{c}',\mylist{\gamma}',\mylist{d}',\mylist{\delta}',b')$.

For terms, given 
\begin{prooftree}
  \AxiomC{$\judge{\Gamma}{a:A}$}
  \AxiomC{$\judge{\Gamma,x:A,\Delta}{b(x):B(x)}$}
  \BinaryInfC{$\judge{\Gamma,\Delta[a/x]}{b(a):B(a)}$}
\end{prooftree}
we have 
\begin{align*}
  b(a)[\mylist{\gamma},\mylist{\delta}] & = b[\mylist{\gamma},a[\mylist{\gamma}],\mylist{\delta}].
\end{align*}

\subsection*{Dependent products}

Given $\judge{\Gamma}{\prod_{x:A}B(x)}$,
\begin{align*}
  \phi\forces_{\mylist{\gamma}} f:\prod_{x:A(\mylist{c})}B(\mylist{c},x)
\end{align*}
if and only if $\phi$ is an operation
\begin{align*}
  \bigl(\alpha\forces_{\mylist{\gamma}}a:A(\mylist{c})\bigr)\quad\longmapsto\quad\bigl(\phi\{\alpha\}\forces_{\mylist{\gamma},\alpha}\app(f,a):B(\mylist{c},a)\bigr)
\end{align*}
such that
\begin{align*}
  \phi\{\alpha\}\cdot\bigl(1_{\mylist{c}},m,\app(f,x)|(1_{\mylist{c}},m)\bigr)
  & = \phi\{\alpha\cdot(1_{\mylist{c}},m)\}
\end{align*}
for $m:a\to a'$ in $|A(\mylist{c})|$.

\subsection*{Dependent sums}

Given $\judge{\Gamma}{\sum_{x:A}B(x)}$,
\begin{align*}
  \upsilon\forces_{\mylist{\gamma}}p:\sum_{x:A(\mylist{c})}B(\mylist{c},x)
\end{align*}
if and only if $\upsilon=(\upsilon_{0},\upsilon_{1})$ with
\begin{align*}
    \upsilon_{0}&\forces_{\mylist{\gamma}}\pi_{0}(p):A(\mylist{c}),\text{
    and}\\
  \upsilon_{1}&\forces_{\mylist{\gamma},\upsilon_{0}}\pi_{1}(p):B(\mylist{c},\pi_{0}(p)).
\end{align*}

\subsection*{Identity types}

Given $\judge{\Gamma,x:A,y:A}{\idn{A}{}(x,y)}$,
\begin{align*}
  \forces_{\mylist{\gamma},\alpha,\beta}f:\idn{A(\mylist{c})}{}(a,b)
  \quad\text{ iff }\quad\alpha\cdot(1_{\mylist{c}},f)=\beta.
\end{align*}

\subsection*{Natural numbers}

Where $\#(t):\nat$ denotes the numeral associated, under the groupoid
interpretation built over the free groupoid, to the closed term
$t:\nat$, we have
\begin{align*}
  \tau\forces t:\nat \quad\text{ iff }\quad\tau:\idn{\nat}{}(t,\# t).
\end{align*}

\newcommand{\SortNoop}[1]{}
\providecommand{\bysame}{\leavevmode\hbox to3em{\hrulefill}\thinspace}
\providecommand{\MR}{\relax\ifhmode\unskip\space\fi MR }
\providecommand{\MRhref}[2]{%
  \href{http://www.ams.org/mathscinet-getitem?mr=#1}{#2}
}
\providecommand{\href}[2]{#2}

\end{document}